\documentclass[11pt]{amsart} \textwidth=14.5cm \oddsidemargin=1cm
\usepackage{amsmath,wasysym}
\usepackage{amsxtra}
\usepackage{amscd}
\usepackage{amsthm}
\usepackage{amsfonts}
\usepackage{amssymb}
\usepackage{eucal}
\usepackage{graphics, color}
\usepackage{mathrsfs}
\usepackage[usenames,dvipsnames]{xcolor}
\usepackage[pdfusetitle,hidelinks,bookmarksdepth=2]{hyperref}

\usepackage{stmaryrd} 

\newtheorem{thm}{Theorem} [section]
\newtheorem{lem}[thm]{Lemma}
\newtheorem{cor}[thm]{Corollary}
\newtheorem{prop}[thm]{Proposition}
\newtheorem{conj}[thm]{Conjecture}

\theoremstyle{definition}
\newtheorem{definition}[thm]{Definition}

\theoremstyle{remark}
\newtheorem{rem}[thm]{Remark}

\numberwithin{equation}{section}

\newcommand{\thmref}[1]{Theorem~\ref{#1}}

\newcommand{\lemref}[1]{Lemma~\ref{#1}}
\newcommand{\propref}[1]{Proposition~\ref{#1}}
\newcommand{\corref}[1]{Corollary~\ref{#1}}

\newcommand{\eqnref}[1]{(\ref{#1})}

\newcommand{\defref}[1]{Definition~\ref{#1}}

\newcommand{\C}{{\mathbb C}}
\newcommand{\N}{{\mathbb N}}

\newcommand{\Q}{\mathbb {Q}}
\newcommand{\Z}{{\mathbb Z}}

\newcommand{\h}{\mathfrak h}

\newcommand{\IND}{\text{IND}\,}
\newcommand{\RES}{\text{RES}\,}

\newcommand{\mf}{\mathfrak}

\newcommand{\U}{\mathbf U}
\newcommand{\ov}{\overline}

\newcommand{\hh}{\mathfrak{H}}

\newcommand{\x}{\mathtt{x}}

\newcommand{\blue}[1]{{\color{blue}#1}}

\newcommand{\qbinom}[2]{\begin{bmatrix}
		#1\\#2
\end{bmatrix} }

\newcommand{\bb}{\mathfrak b}

\DeclareMathOperator{\End}{End}
\DeclareMathOperator{\Mat}{Mat}

\DeclareMathOperator{\Span}{span}
\DeclareMathOperator{\im}{im}
\DeclareMathOperator{\rk}{rk}

\newcommand{\NC}{\mathrm{NC}}
\newcommand{\NH}{\mathrm{NH}}
\newcommand{\Pol}{\mathrm{Pol}}

\newcommand{\tya}[1]{{}^\mathfrak{a}\!#1}
\newcommand{\tyb}[1]{{}^\mathfrak{b}\!#1}
\newcommand{\tyd}[1]{{}^\mathfrak{d}\!#1}
\newcommand{\tyeven}[1]{{#1}}
\newcommand{\tyodd}[1]{{#1}^{\text{--}}}
\newcommand{\aeven}[1]{\tya\tyeven{#1}}
\newcommand{\aodd}[1]{\tya\tyodd{#1}}
\newcommand{\beven}[1]{\tyb\tyeven{#1}}
\newcommand{\bodd}[1]{\tyb\tyodd{#1}}
\newcommand{\deven}[1]{\tyd\tyeven{#1}}
\newcommand{\dodd}[1]{\tyd\tyodd{#1}}

\newcommand{\pbo}{\bodd\partial}

\newcommand{\pdo}{\dodd\partial}

\newcommand{\sbn}{\tyb{s}}

\newcommand{\sbo}{\sbn}  
\newcommand{\sdn}{\tyd{s}}

\newcommand{\sdo}{\sdn}  

\newcommand{\pso}{\tyodd\partial}

\newcommand{\sso}{s}

\newcommand{\elembo}{\bodd{\varepsilon}}
\newcommand{\elemboprime}{{}^\mathfrak{b}\!{\varepsilon}'^{\text{--}}}
\newcommand{\elemdo}{\dodd{\varepsilon}}
\newcommand{\elemdoprime}{\dodd{\varepsilon'}}

\newcommand{\schubo}{{}^\mathfrak{b}\!{\mf{s}}}

\newcommand{\schudo}{{}^\mathfrak{d}\!{\mf{s}}}

\newcommand{\deltab}{{}^\mathfrak{b}\!{\delta}}
\newcommand{\deltad}{{}^\mathfrak{d}\!{\delta}}

\newcommand{\mchbo}{\bodd{\mathcal{H}}}

\newcommand{\mchdo}{\dodd{\mathcal{H}}}

\allowdisplaybreaks

\begin{document}
\title[Spin nilHecke algebras of classical type]{Spin nilHecke algebras of classical type}

\author{Ian T.\ Johnson}

\author{Weiqiang Wang}
\address{Department of Mathematics, University of Virginia, Charlottesville, VA 22904, USA}
\email{ij6fd@virginia.edu (Johnson), ww9c@virginia.edu (Wang)}

\maketitle

\begin{abstract}  
We formulate and study the spin nilHecke algebras ${}^\mathfrak{b}\!{\mathrm{NH}}_n^{\text{--}}$ and ${}^\mathfrak{d}\!{\mathrm{NH}}_n^{\text{--}}$ of type B/D, which differ from the usual nilHecke algebras by some odd signs. 
The type B spin nilHecke algebra is a nil version of the spin type B Hecke algebra introduced earlier by the second author and Khongsap, but not for the type D one. 
We construct faithful polynomial representations $\mathrm{Pol}_n^{\text{--}}$ of the nilHecke algebras via odd Demazure operators. We  formulate the spin Schubert polynomials, and use them to show that the spin nilHecke algebras are matrix algebras with entries in a subalgebra of $\mathrm{Pol}_n^{\text{--}}$ consisting of spin symmetric polynomials.  All these results have their counterparts for the usual nilHecke algebras over the rational field.  Our work is a generalization of results of Lauda and Ellis-Khovanov-Lauda in usual/spin type A. 
\end{abstract}

\setcounter{tocdepth}{1}
\tableofcontents

\section{Introduction}

\subsection{Background}

Affine Hecke algebras and their degenerations \cite{Dr86, Lu89} have many applications in various aspects of representation theory. The nil versions of the degenerate affine Hecke algebras, also known as nilHecke algebras, play a fundamental role in Schubert calculus (cf. \cite{Ku02, FK96}), and the type A nilHecke algebra is a basic ingredient in KLR categorification (cf. \cite{La08}). 

The Schur multiplier (i.e., the second cohomology of a group) arises in projective representations, and the Schur multiplier of an arbitrary finite Weyl group $W$ was computed by Ihara and Yokonuma \cite{IY65} (also cf.  \cite{Kar87}).
Given a 2-cocycle $\alpha$ on $W$, the corresponding twisted (or spin) group algebra $\Q W^\alpha$ admits a Coxeter type presentation, which is almost identical to the standard Coxeter presentation modulo some sign differences; cf. \cite{KW09}. 
The spin (i.e., projective) representation theory of the symmetric groups, equivalently the linear representation theory of the spin symmetric group algebras, was developed by Schur, and rich algebraic combinatorics such as Schur $Q$-functions arises from this.

The spin Hecke algebra of type A (called a degenerate spin affine Hecke algebra then) was introduced by the second author in  \cite[\S3.3]{Wa09}, and it is Morita super-equivalent to the degenerate affine Hecke-Clifford algebra of Nazarov \cite{Na97}. Subsequently the spin Hecke algebras of type B and D were introduced by Khongsap and the second author \cite{KW08}, associated to the ``most nontrivial" 2-cocycle of the corresponding Weyl group. These algebras look almost identical to the degenerate affine Hecke algebras of classical type in \cite{Lu89}, except some odd signs in the defining relations. A remarkable feature is that the polynomial algebras are now replaced by skew-polynomial algebras. The spin and the usual degenerate Hecke algebras have formally the same PBW basis. 

Just as for the degenerate affine Hecke algebras, the spin Hecke algebras admit nil versions as well.
The spin nilHecke algebras of type A, denoted in this paper by $\aodd{\NH}_n$, were rediscovered and studied in depth by Ellis, Khovanov and Lauda in \cite{EKL14}  (called the odd nilHecke algebra  in {\em loc. cit.}). It also reappeared in \cite{KKT16} as a basic building block of a new class of (spin) quiver Hecke superalgebras.

\subsection{The odd/spin type A results}

Let us review some main results of \cite{EKL14} on the spin type A nilHecke algebras $\aodd{\NH}_n$, which are most relevant to our current work.

A faithful (skew-)polynomial representation of $\aodd{\NH}_n$ was constructed via odd Demazure operators. Ellis, Khovanov and Lauda \cite{EKL14} then constructed the ring of odd/spin symmetric polynomials $\aodd{\Lambda}_n$ via odd Demazure operators as a subalgebra of $\aodd{\NH}_n$, and showed that $\aodd{\NH}_n$ is isomorphic to a matrix algebra of size $n!$ with entries in $\aodd{\Lambda}_n$. The sum over all $n$ of Grothendieck groups of $\Z$-graded projective $\aodd{\Lambda}_n$-modules (with the $\Z_2$-grading forgotten), $K_0(\aodd{\NH}) = \bigoplus_{n \geq 0} K_0(\aodd{\NH}_n)$, is shown  to be a twisted bialgebra isomorphic to half the quantum group of rank one $\U_{q}^+(\mathfrak{sl}_2)$.

With the $\Z_2$-grading turned on, this bialgebra isomorphism was subsequently upgraded to an isomorphism with half the quantum covering algebra of rank one $\U_{q,\pi}^+ (\mathfrak{sl}_2)$ in \cite{HW15}, where $\pi$ with $\pi^2=1$ counts the parity $\Z_2$-grading. Note the specialization of $\U_{q,\pi}^+ (\mathfrak{sl}_2)$ at $\pi=1$ becomes $\U_{q}^+ (\mathfrak{sl}_2)$.

All the above results have parallels for the usual  type A nilHecke algebras. The matrix algebra identification for the type A nilHecke algebra was established in \cite{La08}. 

\subsection{The goal}
The goal of this paper is to formulate and establish in the framework of spin type B/D nilHecke algebras generalizations of some main constructions and results (modulo the diagrammatics) of \cite{EKL14} in type A. 

The spin nilHecke algebra $\bodd{\NH}_n$ of type B studied in this paper is exactly the nil version of the corresponding spin Hecke algebra of \cite{KW08}. However the spin nilHecke algebra $\dodd{\NH}_n$ of type D is new as it is {\em not} the nil version of the corresponding spin Hecke algebra therein (which will be denoted by $\dodd{\NH}_{n,\texttt{kw}}$ in this paper). Instead it is associated with a different 2-cocyle of the Weyl group $D_n$; see  \S\ref{subsec:HnH} for the comparison of the two different type D nilHecke algebras. 
These spin type B/D nilHecke algebras are $\Z\times \Z_2$-graded, and they contain as a subalgebra the spin type A nilHecke algebra $\aodd{\NH}_n$. 

%
%
\subsection{The main results}

Let us describe in some detail the main results of this paper section-wise. 

\begin{enumerate}
\item
We construct the (skew-)polynomial representations $\Pol_n^-$ of the spin nilHecke algebras $\bodd{\NH}_n$  and $\dodd{\NH}_n$, where the nilCoxeter generators $\pso_i$ ($1\le i \le n-1$), $\pbo_n$, $\pdo_n$ act by type B/D odd Demazure operators; see Theorems~\ref{thm:rel} and \ref{thm:typedrelations}.

\item
We introduce the rings of spin symmetric polynomials, $\bodd{\Lambda}_n$ and $\dodd{\Lambda}_n$ as the intersections of the kernels of the odd Demazure operators. We show that $\bodd{\Lambda}_n$ and $\dodd{\Lambda}_n$ are polynomial rings in $n$ generators.  See \thmref{thm:bodd} and \propref{prop:dodd}. The rings of spin type B/D symmetric polynomials turn out to be not as odd as in the type A case.

\item
We introduce the spin type B/D Schubert polynomials, parametrized by $B_n$ and $D_n$, respectively. We show that the polynomial representation $\Pol_n^-$ is faithful and the PBW basis theorem holds for the spin nilHecke algebras $\bodd{\NH}_n$  and $\dodd{\NH}_n$. See Propositions~\ref{prop:linearrel} and \ref{prop:linearreld}.

\item
We show that $\Pol_n^-$ is a free $\bodd{\Lambda}_n$-module with these spin Schubert polynomials as a basis; see \propref{prop:bfreemodule}. We establish a similar (slightly weaker) statement in type $D$ over the rational field $\Q$ (instead of being over the ring $\Z$); see \propref{prop:dfreemodule}. We show in \thmref{thm:biso} that the spin type B nilHecke algebra $\bodd{\Lambda}_n$ is isomorphic to a matrix algebra of size $|B_n|$ with entries in the ring of spin symmetric polynomials $\bodd{\Lambda}_n$. For a similar result in type D over $\Q$, see \thmref{thm:diso}.

\item
We show in \propref{prop:module} that $K_0(\bodd{\NH}) = \bigoplus_{n \geq 0} K_0(\bodd{\NH}_n)$ is a bialgebra module over the twisted bialgebra $K_0(\aodd{\NH})$, where the twisted bialgebra $K_0(\aodd{\NH})$ is isomorphic to the quantum covering algebra of rank one $\U_{q,\pi}^+ (\mathfrak{sl}_2)$ \cite{EKL14, HW15}. A similar result holds for type D. 

\item
In Appendix~\ref{appendix:A} we revisit the usual nilHecke algebras associated to arbitrary Weyl groups $W$. The results (1)-(4) hold for nilHecke algebras associated to any Weyl group $W$ over $\Q$. All these are well known, with a possible exception of (4) on  the matrix algebra identification. As we cannot find this explicitly in the literature (except the type A case in \cite{La08}), we offer two proofs, one algebraic and one geometric. The geometric proof was suggested to us by Ben Webster and Peng Shan separately. The algebraic proof is similar to the ones we gave for the spin type D nilHecke algebra. See Remark~\ref{rem:integral} for a possible strengthening over $\Z$ (as in type A \cite{La08}). The type B/D results are occasionally used to provide shortcuts in some proofs in earlier sections in spin nilHecke algebras.
\end{enumerate}

For the convenience of the reader, the different types of spin nilHecke algebras, their Demazure operators, spin symmetric polynomials, and  matrix algebra identifications are summarized in the following Table~\ref{table:spinH}, where we set $\pbo_i =\pdo_i =\pso_i$, for $1\le i\le n-1$. 

\begin{table}[h]
\caption{Spin type A/B/D nilHecke algebras}
\label{table:spinH}
\begin{tabular}{| c | c | c | c |}
\hline
{\Small Type}
& 
{\Small Demazure operators} 
&
{\Small Spin symmetric polynomials}
&
{\Small Spin nilHecke as matrix algebras}
\\
\hline
A
 &
$\pso_i, 0\le i\le n-1$
	&
$\aodd{\Lambda}_n = \bigcap_{i=1}^{n-1} \ker(\pso_i)$
&
$\aodd{\NH}_n  \cong \Mat_{n!}(\aodd{\Lambda}_n)$
\\
\hline
B
& 
$\pbo_i, 0\le i\le n$
&
$\bodd{\Lambda}_n= \bigcap_{i=1}^{n} \ker(\pbo_i)$ 
& 
$\bodd{\NH}_n  \cong \Mat_{|B_n|}(\bodd{\Lambda}_n)$
\\
\hline
D
&
$\pdo_i, 0\le i\le n$
&
$\dodd{\Lambda}_n= \bigcap_{i=1}^{n} \ker(\pdo_i)$ 
&
$\dodd{\NH}_{n,\Q}  \cong \Mat_{|D_n|}  (\dodd{\Lambda}_{n,\Q} )$
\\
\hline
\end{tabular}
\newline
\smallskip
\end{table}

\subsection{Future works}

There are several natural directions to pursue in the theory of spin Hecke algebras.

The spin Hecke algebras \cite{Wa09, KW08} are associated to the most nontrivial 2-cocycles of Weyl groups of classical type. The type D construction in this paper suggests there might exist a more general class of spin Hecke algebras (and double affine versions too) associated to more general 2-cocycles.  

An open basic question is to develop a theory of spin Hecke algebras associated to exceptional Weyl groups. 

Note that our (spin) type B Schubert polynomials are not the ones defined in \cite{FK96, BH95}, and our type B/D Schubert polynomial associated to the longest Weyl group element is a monomial (as in type A). Our definition of Schubert polynomials is crucial in our proof that $\bodd{\NH}_n$ (or its even counterpart) is a matrix algebra over $\Z$, but it may not have a geometric interpretation in terms of cohomology of flag varieties as in \cite{BH95, FK96}. From a combinatorial viewpoint, it will be interesting to see if our version (or another suitable version) of (spin) type B/D Schubert polynomials has additional favorable properties, such as stabilization as $n$ goes to infinity. It will be very interesting to explore spin double Schubert polynomials. 

Lauda and Russell \cite{LR14} developed an intriguing odd Springer theory, building on the spin type A nilHecke algebra and Ellis-Khovanov's theory of odd symmetric polynomials. It will be interesting to see if there is a spin/odd Springer theory of type B and D. 


\vspace{.3cm}

{\bf Acknowledgement.} 
The research of WW  and the undergraduate research of IJ are partially supported by the NSF grants DMS-1405131 and DMS-1702254. We thank Mike Reeks for his help with mentoring IJ in this research. We are also thankful to Ben Webster and Peng Shan for providing a geometric proof of \thmref{thm:iso}.

\section{Spin nilHecke algebras and polynomial representations}
\label{sec:algebradefs}

In this section we introduce the spin nilHecke algebras, $\bodd{\NH}_n$ and $\dodd{\NH}_n$, of  type B and D. 
We construct the (skew-)polynomial representations $\tyodd{\Pol}_n$ of 
the spin nilHecke algebras $\bodd{\NH}_n$ and $\dodd{\NH}_n$ via odd Demazure operators.

\subsection{Spin nilHecke algebras}

We denote by $\tyodd{\Pol}_n$ the skew-polynomial algebra in $n$ variables, that is, the $\Z$-algebra generated by $x_1,\dots,x_n$, subject to the relations:
\begin{equation}
 \label{eq:anti}
x_ix_j + x_jx_i = 0, \quad \forall i \neq j. 
\end{equation}

\begin{definition}
	\label{def:oddtypebnh}
   Let $n\ge 1$.  The {\em spin type B nilHecke algebra} $\bodd{\NH}_n$ is the unital $\Z$-algebra generated by  $x_1,\dots,x_n$ and $\pso_1,\dots,\pso_{n-1},\pbo_n$, subject to the relation \eqref{eq:anti} and the following relations \eqref{eq:arel.1}--\eqref{eq:arel.5} and \eqref{eq:brel.1}--\eqref{eq:brel.5}, for $1\le i\le n-1$:
	\begin{subequations}
		\begin{align}
			(\pso_i)^2 &= 0, \label{eq:arel.1} \\
			\pso_i\pso_{i+1}\pso_i &= \pso_{i+1}\pso_i\pso_{i+1}, \label{eq:arel.2} \\
			\pso_i\pso_j + \pso_j\pso_i &= 0  \quad (|i-j|>1), \label{eq:arel.3} \\
			x_i\pso_i + \pso_i x_{i+1} &= 1,  
			\quad            \pso_i x_i + x_{i+1}\pso_i = 1, \label{eq:arel.4} \\
			 \pso_i x_j + x_j\pso_i &= 0 \quad (j\neq i,i+1);  \label{eq:arel.5}
		\end{align}
	\end{subequations}
	\begin{subequations}
		\begin{align}
			(\pbo_n)^2 &= 0, \label{eq:brel.1} \\
			\pbo_n\pso_{n-1}\pbo_n\pso_{n-1} &= -\pso_{n-1}\pbo_n\pso_{n-1}\pbo_n, \label{eq:brel.2} \\
			\pbo_n\pso_i + \pso_i\pbo_n &= 0 \qquad(1\le i\leq n-2), \label{eq:brel.3} \\
			\pbo_n x_n + x_n\pbo_n &= 1, \label{eq:brel.4} \\
			\pbo_n x_i + x_i\pbo_n &= 0 \qquad (1\le i\leq n-1). \label{eq:brel.5}
		\end{align}
	\end{subequations}

	The {\em spin type B nilCoxeter algebra} $\bodd{\NC}_n$ is defined to be the subalgebra of $\bodd{\NH}_n$ generated by $\pso_i$ for $1 \leq i < n$ and $\pbo_n$.
\end{definition}

\begin{definition} 
	\label{def:oddtypednh}
  Let $n\ge 2$.   The {\em spin type D nilHecke algebra} $\dodd{\NH}_n$ is the unital $\Z$-algebra generated by $x_1,\dots,x_n$ and $\pso_1,\dots,\pso_{n-1},\pdo_n$, subject to the relations \eqref{eq:anti}, \eqnref{eq:arel.1}--\eqnref{eq:arel.5}, and the following additional relations \eqnref{eq:drel.1}--\eqnref{eq:drel.7} for $\pdo_n$:
	\begin{subequations}
		\begin{align}
			(\pdo_n)^2 &= 0, \label{eq:drel.1} \\
			\pdo_n\pso_{n-2}\pdo_n &=   \pso_{n-2}\pdo_n\pso_{n-2}, \label{eq:drel.2} \\
			\pdo_n\pso_i + \pso_i\pdo_n &= 0 \qquad (1\le i \le n-3), \label{eq:drel.3} \\
			\pdo_n\pso_{n-1} - \pso_{n-1}\pdo_n &= 0,   \label{eq:drel.4} \\
			x_{n-1}\pdo_n - \pdo_n x_n  &=  1,   \quad 
			\pdo_n x_{n-1} - x_n\pdo_n  =  1,    \label{eq:drel.6} \\
			\pdo_n x_i + x_i\pdo_n &= 0 \qquad (1\le i \le n-2). \label{eq:drel.7}
		\end{align}
	\end{subequations}
	The {\em spin type D nilCoxeter algebra} $\dodd{\NC}_n$ is defined to be the subalgebra of $\dodd{\NH}_n$ generated by $\pso_i$ for $1 \leq i < n$ and $\pdo_n$.
\end{definition}

	We introduce a $\Z$-grading $|\cdot |$ on the algebra $\bodd{\NH}_n$ by declaring 
	\begin{equation} \label{eq:grading}
	|x_i|=2, \qquad |\pso_j| =|\pbo_n|=-2, 
	\end{equation} 
	for all possible $i,j$.
Similarly, the $\Z$-grading $|\cdot |$ on the algebra  $\dodd{\NH}_n$ is given by declaring $|x_i|=2$, $|\pso_j| =|\pdo_n|=-2$, for all possible $i,j$.
 
We also introduce a parity $\Z_2$-grading $p(\cdot)$ on both $\bodd{\NH}_n$ and $\dodd{\NH}_n$ by declaring all generators $x_i, \pso_i$, $\pbo_n$, and $\pdo_n$ to have parity $1$; that is, $\bodd{\NH}_n$ and $\dodd{\NH}_n$ are naturally superalgebras with all generators being odd.

\subsection{Spin Hecke vs spin nilHecke algebras}
  \label{subsec:HnH}
  
The spin Hecke algebra (also called degenerate spin affine Hecke algebra) of type $B_n$  with 2 parameters $u_1, u_2$, denoted by $\mathfrak H_{B_n}^-$, was introduced in \cite[Definition~4.3]{KW08}; here, we set $u_1=u_2=1$ (see Remark~\ref{rem:unequal} below). Using our current notation the definition of $\mathfrak H_{B_n}^-$ differs from $\bodd{\NH}_n$  only  by substituting the ``nil" relations \eqref{eq:arel.1} and \eqref{eq:brel.1} with the relations
\[
  (\pbo_n)^2 = 1, \qquad      (\pso_i)^2 = 1 \;\; (1\le i \le n-1). 
\]
The spin Hecke algebra $\mathfrak H_{B_n}^-$ is naturally a filtered algebra using the degrees given by \eqref{eq:grading}. Then $\bodd{\NH}_n$  is the associated graded of the filtered algebra $\mathfrak H_{B_n}^-$ (i.e., the nil version of $\mathfrak H_{B_n}^-$). The PBW basis theorem was established in \cite{KW08} for  $\mathfrak H_{B_n}^-$, and so the PBW basis theorem for $\bodd{\NH}_n$ follows; this also follows from our results later in this paper.

Similarly, we also have the spin Hecke algebra of type D, $\mathfrak H_{D_n}^-$, in \cite[Definition~4.1]{KW08}, and it admits a nil version, denoted here by $\dodd{\NH}_{n,\texttt{kw}}$. 
The algebra $\dodd{\NH}_{n,\texttt{kw}}$ is generated by  $x_1,\dots,x_n$ and $\pso_1,\dots,\pso_{n-1},\pdo_n$, subject to the relations \eqref{eq:drel.1}--\eqref{eq:drel.3}, \eqref{eq:drel.7}, and the following relations \eqref{eq:drel.4W}--\eqref{eq:drel.6W} (in place of \eqref{eq:drel.4}--\eqref{eq:drel.6}):
\begin{subequations}
  \begin{align}
   \pdo_n\pso_{n-1} \blue{+} \pso_{n-1}\pdo_n &= 0,   \label{eq:drel.4W} \\
			x_{n-1}\pdo_n \blue{+}  \pdo_n x_n &= 1,   \quad  
			\pdo_n x_{n-1} \blue{+} x_n\pdo_n =   1. \label{eq:drel.6W} 
   \end{align}
\end{subequations}
Note $\dodd{\NH}_{n,\texttt{kw}} \not \cong \dodd{\NH}_n$. Indeed, the action of $\dodd{\NH}_{n,\texttt{kw}}$ on its polynomial representation $\Pol^-_n$ (which is the induced representation from the trivial module of the spin nilCoxter subalgebra) is not faithful (as one checks that $\pdo_n=\pso_{n-1}$), and the action factors through $\aodd{\NH}_n$.  

The finite spin nilCoxeter algebras for $\dodd{\NH}_{n,\texttt{kw}}$ and for $\dodd{\NH}_n$ are associated to distinct 2-cocycles of ${D_n}$ in \cite[Table~2.2]{KW09} (compare \eqref{eq:drel.4} and \eqref{eq:drel.4W}), and hence are non-isomorphic.

Note that $\bodd{\NH}_n$ (and respectively, $\dodd{\NH}_n$, $\dodd{\NH}_{n,\texttt{kw}}$) contains a subalgebra $\aodd{\NH}_n$, which is generated by $x_1,\dots, x_n$, $\pso_1, \dots, \pso_{n-1}$. The algebra  $\aodd{\NH}_n$ is a nil version of the spin Hecke algebra of type $A_{n-1}$ introduced in \cite[\S3.3]{Wa09} and rediscovered  in \cite{EKL14, KKT16}.

\begin{rem}     \label{rem:2D}
By a detailed analysis one can show that up to isomorphism the algebras $\dodd{\NH}_n$ and $\dodd{\NH}_{n,\texttt{kw}}$ are the only possible spin type D nilHecke algebras which contain $\aodd{\NH}_n$ as  a subalgebra. The algebra $\dodd{\NH}_{n,\texttt{kw}}$ will not be considered further in this paper. 
\end{rem}

\begin{rem}
In \cite{Wa09, KW08}, the Hecke-Clifford algebras of classical type were formulated (the definition of type A Hecke-Clifford algebra was due to Nazarov \cite{Na97}) and shown to be Morita super-equivalent to the spin Hecke algebras of the corresponding type. Similar results are valid for the nil versions. Our new nilHecke algebra $\dodd{\NH}_n$ also suggests the existence of spin/Hecke-Clifford algebras associated to more general $2$-cocycles for the finite Weyl groups; cf. \cite[Table~2.2]{KW09}. 
\end{rem}

%

	For a $\Z\times\Z_2$-graded algebra $A$ with a homogeneous basis $\mathfrak B$, we define its graded rank to be (cf. \cite{HW15})
	\begin{equation}  \label{def:grading}
		\rk_{q,\pi}(A) = \sum_{b\in \mathfrak B} q^{|b|}\pi^{p(b)},
	\end{equation}
	where $\pi$ satisfies 
	\[
	\pi^2=1.
	\]
When we need only consider the $\Z$-grading by forgetting the $\Z_2$-grading (or when the $\Z_2$-grading is trivial), we will use the following graded rank:
\[
\rk_q(A)=\rk_{q,1}(A) = \sum_{b\in \mathfrak B} q^{|b|}.
\]

\subsection{Odd Demazure operators of type B}
\label{subsec:demazureB}

We define
 the endomorphisms $\sso_i$ on $\tyodd{\Pol}_n$, for $1 \leq i \leq n-1$, by letting (cf. \cite[(2.2)]{EKL14})
	\begin{equation}
		\label{eq:sso}
		\sso_i(x_j) = \begin{cases}
			-x_{i+1}, &\text{for }j = i \\
			-x_i, &\text{for }j = i+1 \\
			-x_j, &\text{otherwise}.
		\end{cases}
	\end{equation}
In addition, we define
 the endomorphism $\sbo_n$ on $\tyodd{\Pol}_n$ such that 
	\begin{equation}
		\sbo_n(x_j) = -x_j \qquad \forall j.
	\end{equation}
It is straightforward to see that $\sbo_n$ is well defined, i.e., $\sbo_n(x_ix_j + x_jx_i) = 0$ for $i\neq j$.

\begin{lem}
	The operators $\sbo_n$ and $\sso_i$, for $1 \leq i \leq n-1$, satisfy the type $B_n$ Coxeter group relations.
\end{lem}
\begin{proof}
It is known (cf. \cite{EKL14}) that $\sso_i$ ($1 \leq i \leq n-1$) given by \eqref{eq:sso} satisfy the Coxeter relations for $S_n$.
In addition a direct computation shows that
	\begin{align*}
		(\sbo_n)^2 = 1, \qquad 
		\sbo_n\sso_i &= \sso_i\sbo_n \qquad(i \leq n-2), \\
		\sbo_n\sso_{n-1}\sbo_n\sso_{n-1} &= \sso_{n-1}\sbo_n\sso_{n-1}\sbo_n.
	\end{align*}
  The lemma is proved. 
\end{proof}

Now, we are in a position to define the type B odd Demazure operators.

\begin{definition}
	\label{def:pbo}
	The {\em type B odd Demazure operators} $\pso_i$ ($1\leq i\le n-1$) and $\pbo_n$ are defined as $\Z$-linear operators 	on $\tyodd{\Pol}_n$ which satisfy \eqref{eq:pbo}--\eqref{eq:leibnizb:B} below:
	\begin{align}
		\pso_i(1) = 0,  \qquad 
		\pso_i(x_j) &= \begin{cases}
			1, &\text{for }j = i, i+1 \\
			0, &\text{otherwise},
		\end{cases} \label{eq:pbo} \\
		\pbo_n(x_j) &= \begin{cases}
			1, &\text{for }j = n \\
			0, &\text{otherwise},
		\end{cases}
	\end{align}
	and the Leibniz rule:
	\begin{align}  
		\pso_i(fg) &= \pso_i(f)g + \sso_i(f)\pso_i(g),   \label{eq:leibnizb} \\
		\pbo_n(fg) &= \pbo_n(f)g + \sbo_n(f)\pbo_n(g), \quad \forall f, g \in \tyodd{\Pol}_n.     \label{eq:leibnizb:B}
	\end{align}
\end{definition}
Note that $\pso_i$ ($1 \leq i \leq n-1$) are the type A odd Demazure operators defined in~\cite[(2.3)-(2.4)]{EKL14}.

\begin{thm}
	\label{thm:rel}
	The operators $\pso_1,\dots,\pso_{n-1},\pbo_n$ in \eqref{eq:pbo}--\eqref{eq:leibnizb:B}, along with the left multiplication operators $x_1,\dots,x_n$, define a representation of $\bodd{\NH}_n$ on $\tyodd{\Pol}_n$.
\end{thm}

\begin{proof}
	The proof  consists in showing that the relations given in~\defref{def:oddtypebnh} hold.

	The first set of relations~\eqnref{eq:arel.1}--\eqnref{eq:arel.5}, corresponding to type A, have already been proved in~\cite[Proposition~2.1]{EKL14}.
	Thus, we need only prove the remaining relations~\eqnref{eq:brel.1}--\eqnref{eq:brel.5}.

We first prove~\eqnref{eq:brel.4} and~\eqnref{eq:brel.5}, as they will be useful in the proofs of the remaining three relations.
	Let $f \in \tyodd{\Pol}_n$. Then, by the Leibniz rule,
   $\pbo_n(x_n f) = f - x_n\pbo_n(f),$  so that $\pbo_nx_n + x_n\pbo_n = 1$, whence \eqnref{eq:brel.4}.
	Similarly, for $1\le i \le n-1$,
	\begin{equation*}
		\pbo_n(x_i f) = 0 - x_i\pbo_n(f) = x_i\pbo_n(f),
	\end{equation*}
	so that $\pbo_nx_i + x_i\pbo_n = 0$, whence \eqnref{eq:brel.5}.

	To prove~\eqnref{eq:brel.1}--\eqnref{eq:brel.3}, it suffices to prove them in the case where each relation is applied to a monomial; we do so by induction on the degree of the monomial.

	For~\eqnref{eq:brel.1}, the base case is trivial, i.e., $(\pbo_n)^2(1) = 0$.
	For the inductive step, we divide into two cases.
	In the first case, the monomial is of the form $x_n f$, where $f$ is a monomial.
	Then, using~\eqnref{eq:brel.4} and the inductive assumption, we have
	\begin{align*}
		(\pbo_n)^2(x_n f) &= \pbo_n(f - x_n\pbo_n(f)) \\
		&= \pbo_n(f) - \pbo_n(f) + x_n(\pbo_n)^2(f)  
		= x_n(\pbo_n)^2(f) = 0.
	\end{align*}
	In the second case, the monomial is of the form $x_i f$ for $i < n$, and we have
	\begin{equation*}
		(\pbo_n)^2(x_i f) = \pbo_n(-x_i \pbo_n(f)) = x_i(\pbo_n)^2(f) = 0,
	\end{equation*}
	using \eqnref{eq:brel.5} and induction.

	For~\eqnref{eq:brel.2}, we again have a trivial base case.
   There are now three cases, on the first factor in the monomial.
	In the first case, we consider a monomial of the form
	$x_n f$:
	\begin{align*}
		\pbo_n\pso_{n-1}\pbo_n\pso_{n-1}(x_n f) &= \pbo_n\pso_{n-1}\pbo_n(f - x_{n-1}\pso_{n-1}(f)) \\
		&= \pbo_n\pso_{n-1}\pbo_n(f) + \pbo_n\pso_{n-1}(x_{n-1}\pbo_n\pso_{n-1}(f)) \\
		&= \pbo_n\pso_{n-1}\pbo_n(f) + \pbo_n \Big(\pbo_n\pso_{n-1}(f) - x_n\pso_{n-1}\pbo_n\pso_{n-1}(f) \Big) \\
		&= \pbo_n\pso_{n-1}\pbo_n(f) + (\pbo_n)^2\pso_{n-1}(f) \\
		&\qquad - \pso_{n-1}\pbo_n\pso_{n-1}(f) + x_n\pbo_n\pso_{n-1}\pbo_n\pso_{n-1}(f) \\
		&= \pbo_n\pso_{n-1}\pbo_n(f) - \pso_{n-1}\pbo_n\pso_{n-1}(f) + x_n\pbo_n\pso_{n-1}\pbo_n\pso_{n-1}(f),
	\end{align*}
	and
	\begin{align*}
		\pso_{n-1}\pbo_n\pso_{n-1}\pbo_n(x_n f) &= \pso_{n-1}\pbo_n\pso_{n-1}(f - x_n\pbo_n(f)) \\
		&= \pso_{n-1}\pbo_n\pso_{n-1}(f) - \pso_{n-1}\pbo_n \left(\pbo_n(f) - x_{n-1}\pso_{n-1}\pbo_n(f) \right) \\
		&= \pso_{n-1}\pbo_n\pso_{n-1}(f) - \pso_{n-1}(\pbo_n)^2(f) - \pso_{n-1}(x_{n-1}\pbo_n\pso_{n-1}\pbo_n(f)) \\
		&= \pso_{n-1}\pbo_n\pso_{n-1}(f) - \pbo_n\pso_{n-1}\pbo_n(f) + x_n\pso_{n-1}\pbo_n\pso_{n-1}\pbo_n(f).
	\end{align*}
	We arrive at $\pbo_n\pso_{n-1}\pbo_n\pso_{n-1}(x_n f) = -\pso_{n-1}\pbo_n\pso_{n-1}\pbo_n(x_n f)$ by the inductive assumption and the above computations, completing this case.

	In the second case, we consider a monomial of the form $x_{n-1} f$.
	The computations are very similar to the above; for completeness, they are given below.
	\begin{align*}
		\pbo_n\pso_{n-1}\pbo_n\pso_{n-1}(x_{n-1} f) &= \pbo_n\pso_{n-1}\pbo_n(f - x_n\pso_{n-1}(f)) \\
		&= \pbo_n\pso_{n-1}\pbo_n(f) - \pbo_n\pso_{n-1} \left(\pso_{n-1}(f) - x_n\pbo_n\pso_{n-1}(f) \right) \\
		&= \pbo_n\pso_{n-1}\pbo_n(f) - \pbo_n(\pso_{n-1})^2(f) \\
		&\quad + \pbo_n \left(\pbo_n\pso_{n-1}(f) - x_{n-1}\pso_{n-1}\pbo_n\pso_{n-1}(f) \right) \\
		&= \pbo_n\pso_{n-1}\pbo_n(f) + x_{n-1}\pbo_n\pso_{n-1}\pbo_n\pso_{n-1}(f),
	\end{align*}
	and
	\begin{align*}
		\pso_{n-1}\pbo_n\pso_{n-1}\pbo_n(x_{n-1} f) &= \pso_{n-1}\pbo_n\pso_{n-1}(-x_{n-1}\pbo_n(f)) \\
		&= \pso_{n-1}\pbo_n(-\pbo_n(f) + x_n\pso_{n-1}\pbo_n(f)) \\
		&=   \pso_{n-1} \left(\pso_{n-1}\pbo_n(f) - x_n\pbo_n\pso_{n-1}\pbo_n(f) \right) \\
		&= -\pbo_n\pso_{n-1}\pbo_n(f) + x_{n-1}\pso_{n-1}\pbo_n\pso_{n-1}\pbo_n(f).
	\end{align*}
	Again, we arrive at the desired conclusion via induction.

	The final case, 
	$\pbo_n\pso_{n-1}\pbo_n\pso_{n-1}(x_i f) 
	= -\pso_{n-1}\pbo_n\pso_{n-1}\pbo_n(x_if)$, for $i \leq n-2$,  is easily verified. This completes the proof of ~\eqnref{eq:brel.2}.

	Finally, for~\eqnref{eq:brel.3} with  $i \le n-2$, we again induct on the degree of monomials. First consider a monomial of the form $x_n f$.
	Then
	\begin{align*}
		\pbo_n\pso_i(x_n f) &= -\pbo_n(x_n\pso_i(f)) 
		= -\pso_i(f) + x_n\pbo_n\pso_i(f),
	\\
		\pso_i\pbo_n(x_n f) &= \pso_i(f - x_n\pbo_n(f)) 
		= \pso_i(f) + x_n\pso_i\pbo_n(f).
	\end{align*}
	The result follows in this case by induction.
	The other three cases, namely the monomials of the forms $x_i f$, $x_{i+1} f$, and $x_k f$ for $k \neq n, i, i+1$, are similar, and will be skipped. 
\end{proof}
\subsection{Odd Demazure operators of type D}


Define an endomorphism $\sdo_n$ on $\tyodd{\Pol}_n$ by letting
	\begin{equation}
		\sdo_n(x_i) = \begin{cases}
			x_{n-1}, &\text{for }i = n \\
			x_n, &\text{for }i = n-1 \\
			-x_i, &\text{otherwise}.
		\end{cases}
	\end{equation}
It is straightforward to check that $\sdo_n$ is well defined, that is, $\sdo_n(x_ix_j + x_jx_i)=0$ for $i\neq j$.
We also recall the operators $\sso_i$, for $1 \leq i \leq n-1$, from~\eqnref{eq:sso}.

\begin{lem}
	The operators $\sdo_n$ and $\sso_i$ ($1 \leq i \le n-1$) satisfy the type D Coxeter relations.
	\end{lem}
	
\begin{proof}
We already know that $\sso_i$ ($1 \leq i \le n-1$) satisfy the type $A_{n-1}$ Coxeter relations.
It remains to check that 
	\begin{align*}
		(\sdo_n)^2 = 1, \qquad
		\sso_i\sdo_n &= \sdo_n\sso_i \;\; (i\neq n-2), \\
		\sdo_n\sso_{n-2}\sdo_n &= \sso_{n-2}\sdo_n\sso_{n-2}.
	\end{align*}

	The first relation is immediate. For the second relation, if $i < n-2$, we compute
	\begin{equation*}
		\sso_i\sdo_n(x_j) = \sdo_n\sso_i(x_j) = \begin{cases}
			-x_{n-1}, &\text{for }j = n \\
			-x_n, &\text{for }j = n-1 \\
			x_i, &\text{for }j = i+1 \\
			x_{i+1}, &\text{for }j = i \\
			x_j, &\text{otherwise};
		\end{cases}
	\end{equation*}
	and if $i = n-1$, then we compute
	\begin{equation*}
		\sdo_n\sso_{n-1}(x_j) = \sso_{n-1}\sdo_n(x_j) = \begin{cases}
			-x_n, &\text{for }j = n \\
			-x_{n-1}, &\text{for }j = n-1 \\
			x_j, &\text{otherwise}.
		\end{cases}
	\end{equation*}

	Finally, for the third relation, we compute
	\begin{equation*}
		\sdo_n\sso_{n-2}\sdo_n(x_j) = \sso_{n-2}\sdo_n\sso_{n-2}(x_j) = \begin{cases}
			x_{n-2}, &\text{for }j = n \\
			-x_{n-1}, &\text{for }j = n-1 \\
			x_n, &\text{for }j = n-2 \\
			-x_j, &\text{otherwise}.
		\end{cases}
	\end{equation*}
 The lemma is proved. 
\end{proof}


\begin{definition} [Type D odd Demazure operators] 
	We define $\pso_i$, for $1 \leq i \le n-1$, as before by \eqref{eq:pbo} and~\eqref{eq:leibnizb}.
	We also define $\pdo_n$ to be the $\Z$-linear operator of $\tyodd{\Pol}_n$ which satisfies  
	\begin{align}  \label{Demazure:D}
		\pdo_n(1) = 0 \qquad
		\pdo_n(x_j) &= \begin{cases}
			-1, &\text{for }j = n \\
			1, &\text{for }j = n-1 \\
			0, &\text{otherwise},
		\end{cases}
	\end{align}
and the  Leibniz rule
	\begin{equation*}
		\pdo_n(fg) = \pdo_n(f)g + \sdo_n(f)\pdo_n(g), \quad \forall f, g \in \tyodd{\Pol}_n.
	\end{equation*}
\end{definition}

\begin{thm}
	\label{thm:typedrelations}
	The operators $\pso_1,\dots,\pso_{n-1},\pdo_n$ in \eqref{eq:pbo} and \eqref{Demazure:D}, along with the left multiplication operators $x_1,\dots,x_n$, define a representation of $\dodd{\NH}_n$ on $\tyodd{\Pol}_n$.
\end{thm}

\begin{proof}
	The proof consists in showing that the relations given in~\defref{def:oddtypednh} hold.

	Relations~\eqnref{eq:drel.6}--\eqnref{eq:drel.7} are easy consequences of the Leibniz rule.
	The remaining relations are proved by induction on the degree of a monomial.

	For~\eqnref{eq:drel.1}, we assume $(\pdo_n)^2(f) = 0$ and must show that $(\pdo_n)^2(x_j f) = 0$ for $1 \leq j \leq n$.
	For $j \leq n-2$, this follows easily from \eqnref{eq:drel.7}.
	For $j = n-1$, we compute
	\begin{align*}
		(\pdo_n)^2(x_{n-1} f) &= \pdo_n(f + x_n\pdo_n(f)) \\
		&= \pdo_n(f) - \pdo_n(f) + x_{n-1}(\pdo_n)^2(f) 
		= 0.
	\end{align*}
	The computation for $j = n$ is similar, and so we have $(\pso_{n-1})^2 = 0$, whence \eqnref{eq:drel.1}.

	For~\eqnref{eq:drel.2}, the inductive step for a monomial of the form $x_j f$ where $j < n-2$ is again trivial using \eqnref{eq:drel.7}.
	For $j = n-2$, we compute
		\begin{align*}
		\pdo_n\pso_{n-2}\pdo_n(x_{n-2} f) &= -\pdo_n\pso_{n-2}(x_{n-2}\pdo_n(f)) \\
		&= -\pdo_n(\pdo_n(f) - x_{n-1}\pso_{n-2}\pdo_n(f)) \\
		&= \pdo_n(x_{n-1}\pso_{n-2}\pdo_n(f)) \\
		&= \pso_{n-2}\pdo_n(f) + x_n\pdo_n\pso_{n-2}\pdo_n(f),
	\end{align*}
	and
	\begin{align*}
		\pso_{n-2}\pdo_n\pso_{n-2}(x_{n-2} f) &= \pso_{n-2}\pdo_n(f - x_{n-1}\pso_{n-2}(f)) \\
		&= \pso_{n-2}\pdo_n(f) + \pso_{n-2}(-\pso_{n-2}(f) - x_n\pdo_n\pso_{n-2}(f)) \\
		&= \pso_{n-2}\pdo_n(f) + x_n\pso_{n-2}\pdo_n\pso_{n-2}(f),
	\end{align*}
	which agrees with \eqnref{eq:drel.2} in this case via induction.
	For $j = n-1$, we have
	\begin{align*}
		\pdo_n\pso_{n-2}\pdo_n(x_{n-1} f) &= \pdo_n\pso_{n-2}(f + x_n\pdo_n(f)) \\
		&= \pdo_n\pso_{n-2}(f) + \pdo_n(-x_n\pso_{n-2}\pdo_n(f)) \\
		&= \pdo_n\pso_{n-2}(f) + \pso_{n-2}\pdo_n(f) - x_{n-1}\pdo_n\pso_{n-2}\pdo_n(f),
	\end{align*}
	and
	\begin{align*}
		\pso_{n-2}\pdo_n\pso_{n-2}(x_{n-1} f) &= \pso_{n-2}\pdo_n(f - x_{n-2}\pso_{n-2}(f)) \\
		&= \pso_{n-2}\pdo_n(f) + \pso_{n-2}(x_{n-2}\pdo_n\pso_{n-2}(f)) \\
		&= \pso_{n-2}\pdo_n(f) + \pdo_n\pso_{n-2}(f) - x_{n-1}\pso_{n-2}\pdo_n\pso_{n-2}(f),
	\end{align*}
	which again agrees with \eqnref{eq:drel.2} in this case by induction.
	For $j = n$, we compute
	\begin{align*}
		\pdo_n\pso_{n-2}\pdo_n(x_n f) &= \pdo_n\pso_{n-2}(-f + x_{n-1}\pdo_n(f)) \\
		&= -\pdo_n\pso_{n-2}(f) + \pdo_n(\pdo_n(f) - x_{n-2}\pso_{n-2}\pdo_n(f)) \\
		&= -\pdo_n\pso_{n-2}(f) + x_{n-2}\pdo_n\pso_{n-2}\pdo_n(f),
	\end{align*}
	and
	\begin{align*}
		\pso_{n-2}\pdo_n\pso_{n-2}(x_n f) &= \pso_{n-2}\pdo_n(-x_n\pso_{n-2}(f)) \\
		&= \pso_{n-2}(\pso_{n-2}(f) - x_{n-1}\pdo_n\pso_{n-2}(f)) \\
		&= -\pdo_n\pso_{n-2}(f) + x_{n-2}\pso_{n-2}\pdo_n\pso_{n-2}(f),
	\end{align*}
	thus completing the proof of \eqnref{eq:drel.2}.

	For~\eqnref{eq:drel.3}, we must check the monomial $x_j f$ for $j = n-1,n,i,i+1$ (the other cases follow trivially by \eqnref{eq:drel.7}).
	For $j = i$, we have
	\begin{align*}
		\pso_i\pdo_n(x_i f) &= \pso_i(-x_i\pdo_n(f))  
		= -\pdo_n(f) + x_{i+1}\pso_i\pdo_n(f),
	\end{align*}
	and
	\begin{align*}
		\pdo_n\pso_i(x_i f) &= \pdo_n(f - x_{i+1}\pso_i(f))
		= \pdo_n(f) + x_{i+1}\pdo_n\pso_i(f),
	\end{align*}
	verifying \eqnref{eq:drel.3} in this case by induction.
	The case $j = i+1$ is similar.
	For $j = n-1$, we compute
	\begin{align*}
		\pso_i\pdo_n(x_{n-1} f) &= \pso_i(f + x_n\pdo_n(f)) 
		= \pso_i(f) - x_n\pso_i\pdo_n(f),
	\\
		\pdo_n\pso_i(x_{n-1} f) &= \pdo_n(-x_{n-1}\pso_i(f)) 
		= -\pso_i(f) - x_n\pdo_n\pso_i(f),
	\end{align*}
	giving the expected result.
	The case $j = n$ is similar, completing the proof of \eqnref{eq:drel.3}.

	For~\eqnref{eq:drel.4}, we must check the cases $x_{j} f$ with $j = n-1,n$.
	For $j=n-1$, we have
	\begin{align*}
		\pso_{n-1}\pdo_n(x_{n-1} f) &= \pso_{n-1}(f + x_n\pdo_n(f)) \\
		&= \pso_{n-1}(f) + \pdo_n(f) - x_{n-1}\pso_{n-1}\pdo_n(f),
	\end{align*}
	and
	\begin{align*}
		\pdo_n\pso_{n-1}(x_{n-1} f) &= \pdo_n(f - x_n\pso_{n-1}(f)) \\
		&= \pdo_n(f) + \pso_{n-1}(f) - x_{n-1}\pdo_n\pso_{n-1}(f),
	\end{align*}
	verifying the given relation by induction.
	The case $j = n$ is similar. This completes the proof of the theorem. 
\end{proof}

\section{The rings of spin symmetric polynomials}
\label{sec:symmetricpolys}

In this section, we formulate and study the rings of spin symmetric polynomials of type B and D, which are defined via the odd Demazure operators.

%

%
%
\subsection{Spin type B symmetric polynomials}
\begin{lem}
	\label{lem:bexact}
	We have $\im(\pbo_n)=\ker(\pbo_n)$, and $\im(\pso_i)=\ker(\pso_i)$ for $1\leq i\leq n-1.$
\end{lem}

\begin{proof}
	It follows by~\eqnref{eq:brel.1} that $\im(\pbo_n) \subseteq \ker(\pbo_n)$.
	Now suppose that $\pbo_n(f)=0$. Then by~\eqnref{eq:brel.4}, we have
 $ f = (\pbo_n x_n + x_n\pbo_n)f = \pbo_n(x_n f),$ and    so $f\in\im(\pbo_n)$.
 The remaining equalities were shown in \cite{EKL14}, and can be proved similarly as above. 
\end{proof}

	The ring of spin type B symmetric polynomials is defined to be
	\begin{equation}  \label{SF:B}
		\bodd{\Lambda}_n = \bigcap_{i=1}^{n-1} \im(\pso_i) \cap \im(\pbo_n) = \bigcap_{i=1}^{n-1} \ker(\pso_i)  \cap \ker(\pbo_n).
	\end{equation}
The second equality above follows by Lemma~\ref{lem:bexact}. We remark that 
\[
\aodd{\Lambda}_n := \bigcap_{i=1}^{n-1} \im(\pso_i) = \bigcap_{i=1}^{n-1} \ker(\pso_i)
\]
was studied in depth in \cite{EKL14} in connection with $\aodd{\NH}_n$.

The following lemma will be useful later on for computing $\pbo_n$. 
\begin{lem}
	\label{lem:pbo}
	For $k \geq 0$, we have 
	\begin{equation*}
		\pbo_n(x_n^k) = \begin{cases}
			0, &\text{for }k\text{ even} \\
			x_n^{k-1}, &\text{for }k\text{ odd}.
		\end{cases}
	\end{equation*}
\end{lem}
\begin{proof}
	Follows by a simple induction via the Leibniz rule.
\end{proof}

Below (in Lemma~\ref{lem:evenatoodda} and its proof) we find it convenient to use some standard results on the usual (i.e., non-spin) nilHecke algebras $\beven{\NH}_n$ and $\deven{\NH}_n$, or rather on its subalgebra of Weyl group invariant polynomials. These results can be found in Appendix~\ref{appendix:A}, where we describe the nilHecke algebras in general (including the classical type in more detail). We adopt the convention of dropping the superscript $-$ from notations for spin nilHecke algebras and their related constructions to denote their non-spin counterparts. 
We denote by $\Pol_n =\Z[\x_1, \ldots, \x_n]$ the usual polynomial algebra, where the Weyl group of classical type acts naturally. The subalgebra of Weyl group invariant polynomials are denoted by $\tya{\Lambda}_n$, $\tyb{\Lambda}_n$, $\tyd{\Lambda}_n$, respectively. We recall \eqref{eq:invB} here:
\begin{equation*}
\tya{\Lambda}_n = \Z[\x_1, \dots, \x_n]^{S_n},
\quad
\tyb{\Lambda}_n = \Z[\x_1^2, \dots, \x_n^2]^{S_n},
\quad
\tyd{\Lambda}_n = \tyb{\Lambda}_n [\x_1\dotsm \x_n].
\end{equation*}

\begin{lem}
	\label{lem:evenatoodda}
	For any polynomial $f$ in $n$ variables, we have $f(x_1^2,\dotsc,x_n^2) \in \aodd{\Lambda}_n$ if and only if $f(\x_1^2,\dotsc,\x_n^2) \in \aeven{\Lambda}_n$.
\end{lem}

\begin{proof}
It is well known that the subalgebra $\tya{\Lambda}_n$ of the type A Weyl group invariant polynomials  coincides with the intersection of the kernels of the corresponding Demazure operators; cf.~ \eqref{eq:inv}. 

A direct computation in the setting of skew-polynomial representation $\Pol_n^-$ of the type A spin nilHecke algebras gives us 
	\begin{align*}
		s_i(x_j^2) &= \begin{cases}
			x_{i+1}^2, &\text{for }j = i \\
			x_i^2, &\text{for }j = i+1 \\
			x_j^2, &\text{otherwise},
		\end{cases} 
		\qquad\qquad
		\partial_i^-(x_j^2) = \begin{cases}
			x_i - x_{i+1}, &\text{for }j = i \\
			x_{i+1} - x_i, &\text{for }j = i+1 \\
			0, &\text{otherwise},
		\end{cases}
	\end{align*}
A completely analogous computation in the setting of polynomial representation $\Pol_n$ of the usual type A nilHecke algebras gives us
	\begin{align*}
		s_i(\x_j^2) &= \begin{cases}
			\x_{i+1}^2, &\text{for }j = i \\
			\x_i^2, &\text{for }j = i+1 \\
			\x_j^2, &\text{otherwise},
		\end{cases}
		 \qquad \qquad 
		\partial_i(\x_j^2) = \begin{cases}
			\x_i - \x_{i+1}, &\text{for }j = i \\
			\x_{i+1} - \x_i, &\text{for }j = i+1 \\
			0, &\text{otherwise}.
		\end{cases}
	\end{align*}

	As can be seen from the above, the actions of the Demazure operators are formally identical on polynomials of even degree in each variable in the spin and non-spin settings.
	Therefore,  $f(x_1^2,\dotsc,x_n^2) \in \bigcap_{i=1}^{n-1} \ker(\partial_i^-)$ if and only if  $f(\x_1^2,\dotsc,\x_n^2) \in \bigcap_{i=1}^{n-1} \ker(\partial_i)$. The lemma follows. 
\end{proof}

	Define the spin type B elementary symmetric functions, for $1 \le k \le n$:
	\begin{equation}
	\label{def:SFb}
		\elembo_k(x_1,\dotsc,x_n) = \sum_{1\leq i_1 < \dotsb < i_k \leq n} x_{i_1}^2\dotsb x_{i_k}^2.
	\end{equation}

\begin{lem}
	\label{lem:boddelem}
The elements $\elembo_k(x_1,\dotsc,x_n)$ for all $1\le k \le n$ commute with each other. Moreover, we have $\elembo_k(x_1,\dotsc,x_n) \in \bodd{\Lambda}_n$. 
\end{lem}

\begin{proof}
	The commutativity is clear as these elements are of the form $f(x_1^2,\dotsc,x_n^2)$. 
It follows  by~\lemref{lem:evenatoodda} that the elements $\elembo_k(x_1,\dotsc,x_n)$ are in $\aodd{\Lambda}_n$.
	Furthermore, they are also in $\ker(\pbo_n)$ by~\lemref{lem:pbo}. Therefore they are in $\aodd{\Lambda}_n \cap \ker(\pbo_n) =\bodd{\Lambda}_n$ (the equality follows by definition).
\end{proof}

We can now provide a complete description of $\bodd{\Lambda}_n$.

\begin{thm}
	\label{thm:bodd}
We have $\bodd{\Lambda}_n =\Z[x_1^2, \ldots, x_n^2]^{S_n}$, which is a polynomial algebra generated by $\elembo_1, \elembo_2, \ldots, \elembo_n$.
\end{thm}

\begin{proof}
We adapt the proof of~\cite[Proposition~2.2]{EKL14} here. We give some details, as we will repeat the argument for type D later. 
	
 Set $\beven{\Lambda}_n^{-,\text{elem}} :=\Z[x_1^2, \ldots, x_n^2]^{S_n}$, which is well known to be a polynomial algebra generated by $\elembo_1, \elembo_2, \ldots, \elembo_n$. So we have $\beven{\Lambda}_n^{-,\text{elem}} \subset \bodd{\Lambda}_n$ by~\lemref{lem:boddelem}, and we shall prove the equality holds.  
 
{\bf Claim.} Both $\bodd{\Lambda}_n$ and  $\beven{\Lambda}_n^{-,\text{elem}}$ have free abelian group complements in $\Pol_n^-$. 

Let us take the Claim for granted for now. 
Recalling the non-spin type B constructions in Appendix~\ref{appendix:A}, we observe the spin and the usual constructions coincide over the field $\Z_2$. In particular, $\rk_q (\bodd{\Lambda}_n) =\rk_q (\beven{\Lambda}_n)$. 
Now since $\beven{\Lambda}_n^{-,\text{elem}} \subset \bodd{\Lambda}_n$ and both have free complements by the Claim, the graded dimensions over $\Z_2$ of their reductions mod 2 coincide if and only if they are equal.

It remains to prove the Claim. For $\bodd{\Lambda}_n$, this is because if there were no free complement, some free direct summand (as a $\Z$-submodule) would be wholly divisible by an integer $d >1$. But then we could divide generators of this summand by $d$. The result would still be in the kernel of all the odd Demazure operators, a contradiction. As for $\beven{\Lambda}_n^{-,\text{elem}}$, one checks that with respect to a lexicographic order on monomials, the highest order term of the basis of elementary symmetric polynomials always has coefficient 1. 
The Claim (and hence the theorem) is proved.
\end{proof}

	We define the $(q,\pi)$-integers, the $q$-integers,  the $q$-double factorial, and the  $(q,\pi)$-double factorial as follows: 
	\begin{align}        \label{eq:doublefactorial}
	\begin{split}
	   [n] = \frac{q^n-q^{-n}}{q-q^{-1}}, \qquad   &    [2n]!! =[2n] [2n-2] \dotsm[4] [2],
			\\
	  [n]_\pi = \frac{\pi^n q^n-q^{-n}}{\pi q-q^{-1}}, \qquad &  [2n]_\pi!! =[2n]_\pi[2n-2]_\pi\dotsm[4]_\pi[2]_\pi.
		\end{split}
	\end{align}
We have the following corollary to \thmref{thm:bodd}.
\begin{cor}
	\label{cor:bsym2}
	The algebra $\bodd{\Lambda}_n$ has graded rank
	\begin{equation}
		\rk_{q}(\bodd{\Lambda}_n) = q^{-n^2} \frac{1}{(1-q^2)^n} \frac{1}{[2n]!!}.
	\end{equation}
In particular, we have $\rk_q(\bodd{\Lambda}_n)=\rk_q(\beven{\Lambda}_n)$.
\end{cor}

\subsection{Spin type D symmetric polynomials}

\begin{lem}
	\label{lem:dexact}
	We have $\im(\pdo_n)=\ker(\pdo_n)$, and $\im(\pso_i)=\ker(\pso_i)$ for $1\leq i\leq n-1.$
\end{lem}

\begin{proof}
We only need to verify the first identity. By~\eqnref{eq:drel.1}, $\im(\pdo_n) \subseteq \ker(\pdo_n)$.
	If $f\in\ker(\pdo_n)$, then by \eqnref{eq:drel.6} we have that
	$f = -(\pdo_n x_n - x_{n-1}\pdo_n)f = \pdo_n(-x_n f),$ and so $f\in\im(\pdo_n)$.
\end{proof}

	The ring of spin type D symmetric polynomials is defined to be
	\begin{equation} \label{SF:D}
		\dodd{\Lambda}_n = \bigcap_{i=1}^{n-1} \im(\pso_i) \cap \im(\pdo_n) = \bigcap_{i=1}^{n-1} \ker(\pso_i)  \cap \ker(\pdo_n).
	\end{equation}
The second equality above follows by Lemma~\ref{lem:dexact}.

	Define the spin type D elementary symmetric functions:
	\begin{equation}
	 \label{def:SFd}
		\elemdo_k (x_1,\dotsc,x_n) = \begin{cases}
			\sum_{1 \leq i_1 < \dotsb < i_k \leq n} x_{i_1}^2\dotsm x_{i_k}^2, &\text{for }1 \leq k \leq n-1 \\
			x_1 \dotsm x_n, &\text{for }k = n.
		\end{cases}
	   \end{equation}

The following lemma will be useful in considering the spin type D symmetric functions.

\begin{lem}
	\label{lem:pdoevenpowers}
We have
	\begin{equation*}
		\pdo_n(x_i^2) = \begin{cases}
			-x_n - x_{n-1}, &\text{for }i = n \\
			x_n + x_{n-1}, &\text{for }i = n-1 \\
			0, &\text{otherwise}.
		\end{cases}
	\end{equation*}
\end{lem}
\begin{proof}
	Follows by a simple calculation from the definitions.
\end{proof}

\begin{lem}
	\label{lem:DelemSF}
The elements $\elemdo_k$, for $1\le k \le n$, commute with each other. Moreover, we have $\elemdo_k \in \dodd{\Lambda}_n$ for each $k$. 
\end{lem}

\begin{proof}
The commutativity is clear since $\elemdo_k$, for $k \le n-1$, have even degree in each $x_i$. 

We have seen that $\elemdo_k =\elembo_k \in \aodd{\Lambda}_n$, for $1 \leq k \le n-1$. One checks directly that
$\pso_i(\elemdo_n) = 0$, for $1 \leq i \le n-1$, and so $\elemdo_n \in \aodd{\Lambda}_n$. (Alternatively, 
$\elemdo_n =\pm \aodd{\epsilon}_n \in \aodd{\Lambda}_n$ by \cite{EKL14}.)
It remains to show that $\pdo_n(\elemdo_k) = 0$ for $1\leq k\leq n$, since
$\dodd{\Lambda}_n =\aodd{\Lambda}_n \cap \ker(\pdo_n)$.

We first check $\pdo_n(\elemdo_n) = 0$. Indeed, 
	\begin{align*}
		\pdo_n(\elemdo_n) &= \pdo_n(x_1x_2\dotsm x_n) \\
		&= (-1)^{n-2}x_1\dotsm x_{n-2} \pdo_n(x_{n-1}x_n) \\
		&= (-1)^{n-2}x_1\dotsm x_{n-2} (x_n - x_n) = 0.
	\end{align*}
 It follows that $\pbo_n(\elembo_n) = 0$ thanks to $\elembo_n = (-1)^{{n \choose 2}} (\elemdo_n)^2.$ 

We next show $\pdo_n(\elemdo_k) = 0$, for $1 \leq k \le n-1$, by induction on $n$; the base case $n=2$ is trivial using~\lemref{lem:pdoevenpowers}.
Let $n > 2$. Note that, for $1 \leq k \le n-1$,
	\begin{align*}
		\elemdo_k(x_1,\dotsc,x_n) 
		&= \elemdoprime_k(x_2,\dotsc,x_n) + x_1^2\; \dodd{\varepsilon'}_{k-1}(x_2,\dotsc,x_n),
	\end{align*}
where $\dodd{\varepsilon'}_{k-1}$ is the same as $\elemdo_{k-1}$, but with the indices of all variables shifted by 1 as indicated. Using the Leibniz rule and the inductive hypothesis we have
	\begin{align*}
		\pdo_n(\elemdo_k(x_1,\dotsc,x_n)) 
		&= \pdo_n \left(\dodd{\varepsilon'}_k(x_2,\dotsc,x_n) + x_1^2\; \dodd{\varepsilon'}_{k-1}(x_2,\dotsc,x_n) \right) 
		= 0 + x_1^2 \cdot 0 = 0.
	\end{align*} 
The lemma is proved. 
\end{proof}

\begin{prop}
	\label{prop:dodd}
The algebra $\dodd{\Lambda}_n$ is a polynomial algebra generated by $\elemdo_1, \elemdo_2, \ldots, \elemdo_n$.
\end{prop}

\begin{proof}
The same as the proof of~\thmref{thm:bodd}, using \lemref{lem:DelemSF} in place of \lemref{lem:evenatoodda}.
\end{proof}

\begin{cor}
\label{cor:rkSFDn}
We have $\bodd{\Lambda}_n \subset \dodd{\Lambda}_n$, where $\dodd{\Lambda}_n$ has graded rank
	\begin{equation*}
		\rk_q(\dodd{\Lambda}_n) = q^{-n(n-1)} \frac{1}{(1-q^2)^n} \frac{1}{[n][2n-2]!!}.
	\end{equation*}
\end{cor}

\begin{proof}
The inclusion follows by noting $\elembo_k = \elemdo_k$ ($k\neq n$) and $\elembo_n = (-1)^{{n \choose 2}} (\elemdo_n)^2.$ It follows by \propref{prop:dodd} the graded rank is the same as for the usual type D.
\end{proof}

 The following is a nil version of \cite[Proposition~4.6]{KW08}.
\begin{cor}
	\label{cor:bcenter}
 The centers of the algebras $\bodd{\NH}_n$ and $\dodd{\NH}_n$ are $\Z[x_1^2, \ldots, x_n^2]^{S_n}$.
 \end{cor}

\begin{proof}
The quickest way is to refer to \cite[Proposition~2.15]{EKL14}, which states that the center of $\aodd{\NH}_n$ is $\Z[x_1^2, \ldots, x_n^2]^{S_n}$.
As both $\bodd{\NH}_n$ and $\dodd{\NH}_n$ contain $\aodd{\NH}_n$ as a subalgebra, so their centers are included in $\Z[x_1^2, \ldots, x_n^2]^{S_n}$. On the other hand, it is easy to check that each element in $\Z[x_1^2, \ldots, x_n^2]^{S_n}$ commutes with $\pbo_n$ and $\pdo_n$, and so it is central in $\bodd{\NH}_n$ or $\dodd{\NH}_n$. 
\end{proof}

\section{Spin Schubert polynomials of classical type}

In this section we introduce the spin type B/D Schubert polynomials. We compute the Schubert polynomials associated to the identity element of the type B/D Weyl groups as some explicit nonzero constants. 

\subsection{Spin type B Schubert polynomials}
\label{subsec:schubertb}

We denote by $B_n =\langle s_1, \ldots, s_{n-1}, \sbn_n \rangle$ the Weyl group of type $B_n$. 
When there is no confusion, we also write $s_n =\sbn_n$. 
For $w \in B_n$, we choose a reduced expression $w = s_{i_1}\dotsm s_{i_\ell}$ for $w$ in terms of simple transpositions and define $\pbo_w = \pso_{i_1}\dotsm \pso_{i_\ell}$.
A different choice of reduced expression for $w$ gives the same $\pbo_w$ only up to a sign.

We introduce a shorthand $s_{a..b}$ to denote the consecutive product from $s_a$ to $s_b$; similarly, $s_{a..b..c}$ denotes the consecutive product from $s_a$ to $s_b$ and then to $s_c$. For example, we have $s_{1..n..1}=s_1 s_2\dotsm s_{n-1}\sbn_n s_{n-1}\dotsm s_2 s_1$. 

We choose the following reduced expression for the longest element $\tyb{w}_n$ in $B_n$:
\begin{equation}
	\label{w0Bn}
	   w_n = \tyb{w}_n = s_{1..n..1} \cdot s_{2..n..2} \cdot \ldots \cdot s_{(n-1)..n..(n-1)} \cdot
	   \sbn_n.
\end{equation}

For an $n$-tuple of integers  $\underline{r} = (r_1,\dotsc,r_n)$, we write
$\underline{x}^{\underline{r}} = x_1^{r_1}\dotsm x_n^{r_n}.$
Set
	\begin{equation}
		\deltab_n = (2n-1,2n-3,\dots,1),
		\qquad 
		\underline{x}^{\deltab_n} =x_1^{2n-1}x_2^{2n-3} \ldots x_n.
	\end{equation}
We define the spin type B Schubert polynomials to be, for $w \in B_n$,
	\begin{equation}
		\schubo_w(x_1,\dots,x_n) = \pbo_{w^{-1}w_n}(\underline{x}^{\deltab_n}).
	\end{equation}
The following formulas hold, for $w, u \in B_n$:
\begin{equation}
	\label{eq:partial1}
	\pbo_w\pbo_u = \begin{cases}
		\pm\pbo_{wu} & \text{if }\ell(wu) = \ell(w) + \ell(u) \\
		0 &\text{otherwise},
	\end{cases}
\end{equation}
and
\begin{equation}
	\label{eq:partial2}
	\pbo_u \schubo_w = \begin{cases}
		\pm\schubo_{wu^{-1}} & \text{if }\ell(wu^{-1})=\ell(w)-\ell(u) \\
		0 &\text{otherwise}.
	\end{cases}
\end{equation}

Our next goal is to compute $\schubo_e$, where $e \in B_n$ is the identity.

Denote by $\beven{\NH'}_{n-1}^{-}$ the subalgebra of $\bodd{\NH}_n$ generated by $\pso_i$ and $x_i$ for $2\leq  i \leq n$, which is isomorphic to $\bodd{\NH}_{n-1}$. We shall use the prime notation to denote items associated to $\beven{\NH'}_{n-1}^{-}$ systematically, such as  
	\begin{align}  
	 \label{def:shiftedb}
	 \begin{split}
		\tyb{w'}_{n-1} &= s_{2..n..2} \cdot s_{3..n..3}  \cdot \ldots \cdot (s_{n-1}\sbn_n s_{n-1})\cdot \sbn_n,
 \\
		{}^\mathfrak{b}\!{\Pol}'^{\text{--}}_{n-1} &= \Z[ x_2,\dots,x_n] \subset \tyodd{\Pol}_n, 
		\\
		{}^\mathfrak{b}\!{\Lambda}'^{\text{--}}_{n-1} &= \bigcap_{i=2}^n \ker (\pso_i: {}^\mathfrak{b}\!{\Pol}'^{\text{--}}_{n-1} \rightarrow {}^\mathfrak{b}\!{\Pol}'^{\text{--}}_{n-1}), 
		\\
		{}^\mathfrak{b}\!{\delta}'^{\text{--}}_{n-1} &= (2n-3,  \dots,3,1). 
		\end{split}
	\end{align}

The following is a type A analogue of~\lemref{lem:pbo}.

\begin{lem}
	\label{lem:pso}
	For $1\leq i < n$, we have
	\begin{equation}
		\pso_i(x_i^k) = \sum_{j=1}^k (-1)^{j-1} x_{i+1}^{j-1} x_i^{k-j},
	\end{equation}
	\begin{equation}
		\pso_i(x_{i+1}^k) = \sum_{j=1}^k (-1)^{j-1} x_i^{j-1} x_{i+1}^{k-j}.
	\end{equation}
\end{lem}
\begin{proof}
It follows by a simple induction on $k$ and the Leibniz rule for $\pso_i$. 
\end{proof}

\begin{prop}
	\label{prop:schubo}
	We have $\schubo_e = \pm1$.
\end{prop}

\begin{proof}
We shall use a shorthand notation similar to \eqref{w0Bn} such as
\[
\pso_{1..n..1}=\pso_1 \pso_2\dotsm \pso_{n-1}\pbo_n \pso_{n-1}\dotsm \pso_2 \pso_1.
\] 

We proceed by induction on $n$, with the base case $n=1$ being clear. 

Assume $n>1$.
	Then
	\begin{align*}
		\schubo_e &= \pbo_{w_n}(\underline{x}^{\deltab_n}) \\
		&= \pso_{1..n..1} \pbo_{w'_{n-1}}(x_1^{2n-1}\underline{x}^{\bodd{\delta'}_{n-1}}) \\
		&= \pso_{1..n..1} (-x_1^{2n-1}\pbo_{w'_{n-1}}(\underline{x}^{\bodd{\delta'}_{n-1}})) \\
		&= \pm\pso_{1..n..2} \pso_{1} (x_1^{2n-1}) \\
		&= \pm\pso_{1..n..2} \left( \sum_{j=1}^{2n-1} (-1)^{j-1} x_2^{j-1}x_1^{2n-j-1} \right) \\
		&= \pm\pso_1\sum_{j=1}^{2n-1} \pso_{2..n..2} \left( (-1)^{j-1} x_2^{j-1} \right) x_1^{2n-j-1}.
	\end{align*}
	Now, the expression $\pbo_2\dotsm\pso_{n-1}\pbo_n\pso_{n-1}\dotsm\pbo_2$ consists of $2n-3$ Demazure operators, each of which will decrease the degree of a polynomial by  1.
Hence the only terms in the above sum which will survive are $j=2n-2,2n-1$, which leads to the following simplification:
	\begin{align*}
		\schubo_e &= \pm\pso_1\sum_{j=2n-2}^{2n-1} \pso_{2..n..2} \left( (-1)^{j-1} x_2^{j-1} \right) x_1^{2n-j-1} \\
		&= \pm\pso_1\sum_{j_1=2n-2}^{2n-1} \pso_{2..n..3}
		 \left( (-1)^{j_1-1} \sum_{j_2=1}^{j_1-1} (-1)^{j_2-1}x_3^{j_2-1}x_2^{j_1-j_2-1} \right) x_1^{2n-j_1-1} \\
		&= \pm\pso_1\sum_{j_1=2n-2}^{2n-1} \pso_2\sum_{j_2=2n-4}^{j_1-1} \pso_{3..n..3}
		\left( (-1)^{j_1-1} (-1)^{j_2-1}x_3^{j_2-1}x_2^{j_1-j_2-1} \right) x_1^{2n-j_1-1}.
	\end{align*}
	Continuing this process,
	\begin{align*}
		\schubo_e &= \pm \sum_{j_1=2n-2}^{2n-1}\sum_{j_2=2n-4}^{j_1-1}\dots\sum_{j_{n-1}=2}^{j_{n-2}-1} \\
		&\qquad\qquad \pso_1\dotsm\pbo_n \left( (-1)^{j_1+\dots+j_{n-1}+n-1} x_n^{j_{n-1}-1} x_{n-1}^{j_{n-2}-j_{n-1}-1} \dotsm x_2^{j_1-j_2-1} x_1^{2n-j_1-1} \right) \\
		&= \pm \sum_{j_1=2n-2}^{2n-1}\sum_{j_2=2n-4}^{j_1-1}\dots\sum_{j_{n-2}=4}^{j_{n-3}-1}\sum_{\substack{j_{n-1}=2 \\ j_{n-1}\text{ even}}}^{j_{n-2}-1} \\
		&\qquad\qquad \pso_1\dotsm\pso_{n-1} \left( (-1)^{j_1+\dots+j_{n-2}+n+1} x_n^{j_{n-1}-2} x_{n-1}^{j_{n-2}-j_{n-1}-1} \dotsm x_2^{j_1-j_2-1} x_1^{2n-j_1-1} \right).
	\end{align*}

	Now, we factor all $x_i$ for $i\leq n-2$ to the right of $\pso_{n-1}$, and consider only the expression $\pso_{n-1}(x_n^{j_{n-1}-2}x_{n-1}^{j_{n-2}-j_{n-1}-1})$.
	After expanding this expression, any monomial terms with nonzero powers of $x_n$ will be annihilated by $\pso_1\dotsm\pso_{n-2}$ for degree reasons and hence can be ignored.
	Thus, we can use the Leibniz rule combined with~\lemref{lem:pso} to simplify this expression: for $j_{n-1}-2\neq 0,1$ and $j_{n-2}-j_{n-1}-1\neq 0,1$, we have (ignoring all terms with nonzero powers of $x_n$)
	\begin{align*}
 \pso_{n-1} & (x_n^{j_{n-1}-2}x_{n-1}^{j_{n-2}-j_{n-1}-1}) \\
		&= \pso_{n-1}(x_n^{j_{n-1}-2})x_{n-1}^{j_{n-2}-j_{n-1}-1} 
		 + (-1)^{j_{n-1}-2}x_{n-1}^{j_{n-1}-2} \pso_{n-1}(x_{n-1}^{j_{n-2}-j_{n-1}-1}) \\
		&\equiv (-1)^{j_{n-1}-3}x_{n-1}^{j_{n-2}-4} + (-1)^{j_{n-1}-2}x_{n-1}^{j_{n-2}-4} 
		= 0\; (\text{modulo monomials involving } x_n).
   \end{align*}
So we need only consider the cases when $j_{n-1}-2 = 0$ or $j_{n-2}-j_{n-1}-1 = 0,1$ (recall $j_{n-1}$ is always even). 
	It is also readily checked that the case where $j_{n-1}-2 > 0$ and $j_{n-2}-j_{n-1}-1=1$, we obtain the same result as above (that the expression equals zero).
	Additionally, the cases $j_{n-1}-2 = 0$ and $j_{n-2}-j_{n-1}-1 = 0$ are mutually exclusive, since $j_{n-2}\geq 4$ in the sum.
	In the case $j_{n-1}-2=0$, we obtain (regardless of the value of $j_{n-2}-j_{n-1}-1$, since it cannot be zero)
	\begin{equation}
		\label{eq:schubo1}
		\pso_{n-1}(x_{n-1}^{j_{n-2}-3}) \equiv x_{n-1}^{j_{n-2}-4} \; (\text{modulo monomials involving } x_n).
	\end{equation}
	In the case $j_{n-2}-j_{n-1}-1=0$, we obtain
	\begin{equation}
		\label{eq:schubo2}
		\pso_{n-1}(x_n^{j_{n-2}-3}) \equiv (-1)^{j_{n-2}-4}x_{n-1}^{j_{n-2}-4} = -x_{n-1}^{j_{n-2}-4}
		\; (\text{modulo monomials involving } x_n),
	\end{equation}
	since $j_{n-1}$ is always even and so $j_{n-2}=j_{n-1}+1$ is odd.

	Now, considering the terms in the sum for a particular value of $j_{n-2}$, we will obtain a contribution from~\eqnref{eq:schubo1} from the term $j_{n-1}=2$; if $j_{n-2}$ is odd, we will also obtain a \emph{distinct} contribution from~\eqnref{eq:schubo2}, which cancels the first contribution.
	Therefore, only even values of $j_{n-2}$ contribute anything to the sum, and we can write
	\begin{align*}
		\schubo_e &= \pm \sum_{j_1=2n-2}^{2n-1}\sum_{j_2=2n-4}^{j_1-1}\dots\sum_{j_{n-3}=6}^{j_{n-4}-1}\sum_{\substack{j_{n-2}=4 \\ j_{n-2}\text{ even}}}^{j_{n-3}-1} \\
		&\qquad\quad \pso_1\dotsm\pbo_{n-2} \left( (-1)^{j_1+\dots+j_{n-3}+n+1} x_{n-1}^{j_{n-2}-2} x_{n-2}^{j_{n-3}-j_{n-2}-1} \dotsm x_2^{j_1-j_2-1} x_1^{2n-j_1-1} \right).
	\end{align*}
	This expression is of the same form as before, and we repeat this same procedure $n-2$ more times to arrive at  
	\begin{equation*}
		\schubo_e 
		= \pm(-1)^{n+1} = \pm1.
	\end{equation*}
The proposition is proved. 
\end{proof}

\begin{lem}
	\label{lem:schubertlen}
	Let $w, u \in B_n$. If $\ell(w) < \ell(u)$, then $(\underline{x}^{\underline{r}}\pbo_u)(\schubo_w) = 0$. Moreover, if $\ell(w) = \ell(u)$, then
	\begin{equation*}
		(\underline{x}^{\underline{r}}\pbo_u)(\schubo_w) = \begin{cases}
			\pm\underline{x}^{\underline{r}} & \text{ if }w=u \\
			0 &\text{otherwise}.
		\end{cases}
	\end{equation*}
\end{lem}
\begin{proof}
By~\eqnref{eq:partial1}--\eqnref{eq:partial2} we have 
   \begin{equation}  \label{eq:dsw}
	  \pbo_u (\schubo_w) = \begin{cases}
			\pm \schubo_{wu^{-1}}  & \text{ if } \ell(wu^{-1}) =\ell(w) -\ell(u) \\
			0 &\text{otherwise}.
		\end{cases}
	\end{equation}
  The lemma follows from \eqref{eq:dsw}.
\end{proof}


\begin{prop}
	\label{prop:linearrel}
	There are no linear relations among the images of ${\{\underline{x}^{\underline{r}}\, \pbo_w\}}_{w\in B_n,\underline{r}\in\N^n}$ or among those of ${\{\pbo_w\, \underline{x}^{\underline{r}}\}}_{w\in B_n,\underline{r}\in\N^n}$ in $\End(\tyodd{\Pol}_n)$. Thus these two sets form $\Z$-bases for $\bodd{\NH}_n$.
\end{prop}

\begin{proof}
 The proof here is fairely standard using ~\lemref{lem:schubertlen}. 
 
 Note these two sets are spanning sets for $\bodd{\NH}_n$ by the defining relations of $\bodd{\NH}_n$. It suffices to prove the linear independence of either of these two sets, and we choose to prove that ${\{\underline{x}^{\underline{r}}\, \pbo_w\}}_{w\in B_n,\underline{r}\in\N^n}$  is linearly independent. 
 
 Assume we have a nontrivial relation $S :=\sum_{u\in B_n,\underline{r}\in\N^n} c_{u, \underline{r}} \underline{x}^{\underline{r}}\, \pbo_u =0$ for some scalars $c_{u, \underline{r}}$, and $w$ is of minimal length such that $c_{w, \underline{r}'} \neq 0$ for some $\underline{r}'$. By ~\lemref{lem:schubertlen}, we have  $0= S (\schubo_w) =\sum_{\underline{r}} \pm c_{w, \underline{r}} \underline{x}^{\underline{r}}$, which is a contradiction.
\end{proof}

Recall the notion of $(q,\pi)$-double factorial $[2n]_{\pi}!!$ from  \eqref{eq:doublefactorial}. 
\begin{cor}
	\label{cor:bfaithful}
	The representation of the spin type B nilHecke algebra $\bodd{\NH}_n$ on $\tyodd{\Pol}_n$ is faithful. Moreover, we have the following graded rank formulas:
	\begin{align*}
		\rk_{q,\pi}(\bodd{\NC}_n) &= (\pi q)^{-n^2}[2n]_\pi !!, \\
		\rk_{q,\pi} (\bodd{\NH}_n) &= \frac{(\pi q)^{-n^2}[2n]_{\pi} !!}{{(1-\pi q^2)}^n}.
	\end{align*}
\end{cor}
\begin{proof}
The faithfulness is a simple consequence of \propref{prop:linearrel}.

The first graded rank formula follows from the definition of the $\Z$-grading 
	\begin{equation*}
		\rk_{q,\pi}(\bodd{\NC}_n) = \sum_{w\in B_n} \pi^{\ell(w)} q^{-2\ell(w)} = (\pi q)^{-n^2}[2n]_\pi !!.
	\end{equation*}
	The second formula follows from the above and that $ \rk_{q,\pi}(\Pol^-_n) =\frac{1}{{(1-\pi q^2)}^n}.$
 \end{proof}
The above formula can be compared with the graded rank formula for the nilHecke algebra of type A \cite[(5.16)]{HW15}:
   \begin{align*}
		\rk_{q,\pi} (\aodd{\NH}_n) &= \frac{(q\pi)^{-{n \choose 2}}[n]_{\pi}!}{{(1-\pi q^2)}^n}.
	\end{align*}
\subsection{Spin type D Schubert polynomials}

Let $n\ge 2$. 
Let $D_n=\langle s_1, \ldots, s_{n-1}, \sdn_n \rangle$ denote the Weyl group of type D. Sometimes we write $s_n =\sdn_n$. 
For $w \in D_n$,  we choose a reduced expression $w = s_{i_1}\dotsm s_{i_{\ell}}$ in terms of simple transpositions and define $\pdo_w = \pdo_{i_1}\dotsm\pdo_{i_{\ell}}$.
We  consider the following reduced expression of the longest word $\dodd{w}_n$ of $D_n$:
\begin{equation}
	\begin{split}
	  w_n=  \dodd{w}_n &= s_{1..(n-2) n..1} \cdot s_{2..(n-2) n..2} \cdot \ldots \cdot (s_{n-2}\sdn_n).
	\end{split}
\end{equation}
Set
	\begin{equation}
	 \label{eq:deltad}
		\deltad_n = (2n-2,2n-4,\dots,2,0),
		\qquad 
		\underline{x}^{\deltad_n} =x_1^{2n-2}x_2^{2n-4} \ldots  x_{n-1}^2.
	\end{equation}
For $w \in D_n$, we define the spin type D Schubert polynomials
	\begin{equation}
		\schudo_w(x_1,\dots,x_n) = \pdo_{w^{-1}w_n}(\underline{x}^{\deltad_n}).
	\end{equation}
As in type B, we have
\begin{equation*}
	\pdo_w\pdo_u = \begin{cases}
		\pm\pdo_{wu} & \ell(wu)=\ell(w)+\ell(u) \\
		0 &\text{otherwise},
	\end{cases}
\end{equation*}
and
\begin{equation}
\label{eq:Dss}
	\pdo_u (\schudo_w) = \begin{cases}
		\pm \schudo_{wu^{-1}} & \ell(wu^{-1})=\ell(w)-\ell(u) \\
		0 &\text{otherwise}.
	\end{cases}
\end{equation}

The following lemma is a generalization of~\lemref{lem:pdoevenpowers}.

\begin{lem}
	\label{lem:pdo}
	For any $k\geq1$, we have
	\begin{align*}
		\pdo_n(x_n^k) &= -\sum_{j=1}^k x_{n-1}^{j-1} x_n^{k-j},
\\
		\pdo_n(x_{n-1}^k) &= \sum_{j=1}^k x_n^{j-1} x_{n-1}^{k-j}.
	\end{align*}
\end{lem}
\begin{proof}
It follows by an induction on $k$ and the Leibniz rule. 
\end{proof}

\begin{prop}
	\label{prop:schudo}
	We have $\schudo_e = 2^{n-1}$.
\end{prop}

\begin{proof} 
We proceed by induction on $n$.
	In the base case $n=2$, we have
	\begin{align*}
		\schudo_e &= \pdo_{w_2}(\underline{x}^{\deltad_2}) 
		= \pso_1\pdo_2(x_1^2)  
		= \pso_1(x_1 + x_2) 
		= 2.
	\end{align*}

	For $n>2$, using inductive assumption we have
	\begin{align*}
		\schudo_e = \pdo_{w_n}(\underline{x}^{\deltad_n}) 
		&= \pso_{1..(n-2)n..1}  \pdo_{w'_{n-1}} \Big(x_1^{2n-2}\underline{x}^{\dodd{\delta'}_{n-1}} \Big) \\
		&= \pso_{1..(n-2)n..1}  (x_1^{2n-2}\schudo'_e) \\
		&= 2^{n-2} \pso_{1..(n-2)n..1}  (x_1^{2n-2}).
	\end{align*}
  Again, we use~\lemref{lem:pso} and a similar trick as in the proof of~\propref{prop:schubo} to simplify the sums involved:
	\begin{align*}
		\schudo_e 
		&= 2^{n-2}\pso_{1..(n-2)n..2}  \left( \sum_{j_1=2n-3}^{2n-2} (-1)^{j_1-1}x_2^{j_1-1}x_1^{2n-j_1-2} \right)
		 \\
		&= 2^{n-2} \pso_{1..(n-2)n..3}  \left( \sum_{j_1=2n-3}^{2n-2} \sum_{j_2=2n-5}^{j_1-1} (-1)^{j_1+j_2-2} x_3^{j_2-1}x_2^{j_1-j_2-1}x_1^{2n-j_1-2} \right) \\
		&= 2^{n-2} \sum_{j_1=2n-3}^{2n-2} \sum_{j_2=2n-5}^{j_1-1} \dots \sum_{j_{n-1}=1}^{j_{n-2}-1} \\
		&\qquad\qquad \pso_{1..(n-2)}\pdo_n \left( (-1)^{j_1+\dots+j_{n-1}+n-1} x_n^{j_{n-1}-1}x_{n-1}^{j_{n-2}-j_{n-1}-1} \dotsm x_2^{j_1-j_2-1}x_1^{2n-j_1-2} \right).
	\end{align*}

	As in the proof of~\propref{prop:schubo}, we want to evaluate $\pdo_n(x_n^{j_{n-1}-1}x_{n-1}^{j_{n-2}-j_{n-1}-1})$, ignoring any resulting monomial terms involving $x_n$ which will be annihilated by $\pso_{1..(n-2)}$.
	Thus, if we have $j_{n-1}-1,j_{n-2}-j_{n-1}-1\neq0$, we can use~\lemref{lem:pdo} to compute
	\begin{align*}
		\pdo_n(x_n^{j_{n-1}-1}x_{n-1}^{j_{n-2}-j_{n-1}-1}) 
		&\equiv -x_{n-1}^{j_{n-1}-2}x_{n-1}^{j_{n-2}-j_{n-1}-1} + x_{n-1}^{j_{n-1}-1}x_{n-1}^{j_{n-2}-j_{n-1}-2} \\
		&= 0 \; (\text{modulo monomials involving } x_n),
	\end{align*}
	thus leaving only the terms with $j_{n-1}=1$ and $j_{n-1}=j_{n-2}-1$ (note that these cases are mutually exclusive).
	In the former case $j_{n-1}=1$, we have
	\begin{equation*}
		\pdo_n(x_{n-1}^{j_{n-2}-2}) = x_{n-1}^{j_{n-2}-3},
	\end{equation*}
	with a leading coefficient of
	$(-1)^{j_1+\dots+j_{n-1}+n-1} = (-1)^{j_1+\dots+j_{n-2}+n}$; in the latter case $j_{n-1}=j_{n-2}-1$, we have
	\begin{equation*}
		\pdo_n(x_n^{j_{n-2}-2}) = -x_{n-1}^{j_{n-2}-3},
	\end{equation*}
	with a leading coefficient of
   $(-1)^{j_1+\dots+j_{n-1}+n-1} = (-1)^{j_1+\dots+j_{n-3}+n}.$    Thus, these two distinct terms will cancel each other when $j_{n-2}$ is even, and will combine when $j_{n-2}$ is odd.
   This gives us  
  \begin{align*}
		\schudo_e &= 2^{n-2} \sum_{j_1=2n-3}^{2n-2} \sum_{j_2=2n-5}^{j_1-1} \dots \sum_{j_{n-3}=5}^{j_{n-4}-1} \sum_{\substack{j_{n-2}=3 \\ j_{n-2}\text{ odd}}}^{j_{n-3}-1} \\
		&\qquad\qquad \pso_1\dotsm\pso_{n-2} \left( -2(-1)^{j_1+\dots+j_{n-3}+n} x_{n-1}^{j_{n-2}-3}x_{n-2}^{j_{n-3}-j_{n-2}-1} \dotsm x_2^{j_1-j_2-1}x_1^{2n-j_1-2} \right).
	\end{align*}

	We can now apply a similar observation to the above using~\lemref{lem:pso} (which is effectively the same usage as in the proof of~\propref{prop:schubo}) to obtain
	\begin{align*}
		\schudo_e &= -2 \cdot 2^{n-2} \sum_{j_1=2n-3}^{2n-2} \sum_{j_2=2n-5}^{j_1-1} \dots \sum_{j_{n-4}=7}^{j_{n-5}-1} \sum_{\substack{j_{n-3}=5 \\ j_{n-3}\text{ odd}}}^{j_{n-4}-1} \\
		&\qquad\qquad \pso_1\dotsm\pso_{n-3} \left( (-1)^{j_1+\dots+j_{n-4}+n+1} x_{n-2}^{j_{n-3}-5}x_{n-3}^{j_{n-4}-j_{n-3}-1} \dotsm x_2^{j_1-j_2-1}x_1^{2n-j_1-2} \right) \\
		& 
		= -2\cdot 2^{n-2} (-1)^{2n-3}
		= 2^{n-1}.
	\end{align*}
The proposition is proved.
\end{proof}

The obvious type D counterparts of~\eqnref{eq:partial1}, \eqnref{eq:partial2}  and \eqref{eq:dsw} remain to be valid. Together with $\schudo_e = 2^{n-1}$ (see \propref{prop:schudo}), these imply the following type D counterpart of ~\lemref{lem:schubertlen}. 

\begin{lem}
	\label{lem:schubertlend}
	Let $w, u \in D_n$. If $\ell(w) < \ell(u)$, then $(\underline{x}^{\underline{r}}\pdo_u)(\schudo_w) = 0$. Moreover, if $\ell(w) = \ell(u)$, then
	\begin{equation*}
		(\underline{x}^{\underline{r}}\pdo_u)(\schudo_w) = \begin{cases}
			\pm 2^{n-1}   \underline{x}^{\underline{r}} & \text{ if } w=u  \\
			0 &\text{otherwise}.
		\end{cases}
	\end{equation*}
\end{lem}

\begin{prop}
	\label{prop:linearreld}
	There are no linear relations among the images of ${\{\underline{x}^{\underline{r}}\pdo_w\}}_{w\in D_n,\underline{r}\in\N^n}$ or among those of ${\{\pdo_w\underline{x}^{\underline{r}}\}}_{w\in D_n,\underline{r}\in\N^n}$ in $\End(\tyodd{\Pol}_n)$. Thus these two sets form $\Z$-bases for $\dodd{\NH}_n$.
\end{prop}

\begin{proof}
The proof is identical to the one for~\propref{prop:linearrel}.
\end{proof}

\begin{cor}
	\label{cor:dfaithful}
	The action of the spin type D nilHecke algebra on $\tyodd{\Pol}_n$ is faithful. We have the following graded rank formulas:   
	\begin{align*}
		\rk_{q,\pi}(\dodd{\NC}_n) &= (\pi q)^{-n(n-1)}[n]_{\pi} [2n-2]_{\pi}!!, \\
		\rk_{q,\pi}(\dodd{\NH}_n) &= \frac{(\pi q)^{-n(n-1)}[n]_{\pi}[2n-2]_{\pi}!! }{{(1-\pi q^2)}^n}.
	\end{align*}
\end{cor}

\begin{proof}
	We have the following identity:
	 \begin{equation*}
		\rk_{q,\pi}(\dodd{\NC}_n) = \sum_{w\in D_n} \pi^{\ell(w)} q^{-2\ell(w)} = (\pi q)^{-n(n-1)}[n]_{\pi} [2n-2]_{\pi}!!.
	\end{equation*}
		The rest of the proof is the same as for \corref{cor:bfaithful}.
\end{proof}

\section{Spin nilHecke algebras as matrix algebras}

In this section we show that $\Pol_n^-$ is a free $\bodd{\Lambda}_n$-module with a basis of spin Schubert polynomials, and then show that $\bodd{\NH}_n$ is a matrix algebra over $\bodd{\Lambda}_n$ of size $2^n n!$. We also show that after a base change to $\Q$, $\dodd{\NH}_{n,\Q}$ is a matrix algebra over $\dodd{\Lambda}_{n,\Q}$ of size $2^{n-1} n!$.
Finally, we show the spin nilHecke algebras of classical type  provide a categorification of a bialgebra module over the quantum covering algebra of rank one.

\subsection{The spin type B case}

Let
	\begin{equation}
		\begin{split}
			\mchbo_n &= \Span_{\Z}\{\underline{x}^{\underline{r}} \in \tyodd{\Pol}_n \mid \underline{r} \leq \delta_n\text{ termwise}\} \\
			&= \Span_{\Z}\{x_1^{r_1}\dotsm x_n^{r_n} \mid r_i \leq 2n-2i+1 \text{ for }1 \leq i \leq n\}.
		\end{split}
	\end{equation}

\begin{lem}
	\label{lem:mchbo}
	The spin type B Schubert polynomials ${\{\schubo_w\}}_{w\in B_n}$ form a $\Z$-basis for $\mchbo_n$.
\end{lem}

\begin{proof}
%
	It follows immediately from their definition that the spin type B Schubert polynomials are all contained in $\mchbo_n$.
	Both $\mchbo_n$ and the set of spin type B Schubert polynomials have $(2n)!! $ elements.
	Thus, if we have
	\begin{equation*}
		\sum_{w\in B_n} c_w\schubo_w(x) = 0
	\end{equation*}
	for some $c_w \in \Q$, we can apply the operators $\pbo_u$ as in the proof of~\propref{prop:linearrel} to deduce that all $c_w = 0$.
	In other words, we pick out a longest word $w$ such that $c_w \neq 0$  and apply $\pbo_{w^{-1}}$ to obtain (by~\lemref{lem:schubertlen}) that $\pm c_w = 0$, a contradiction.
	This proves linear independence over $\Q$.

	Now, if we have an expression $f = \sum_w c_w \schubo_w$ for $f \in \mchbo_n$, we similarly take a word $w$ of maximal length such that $c_w \in \Q\setminus\Z$ and apply $\pbo_{w}$ to get $\pm c_w = \pbo_{w} f$. But since $f$ has integral coefficients, so does $\pbo_{w^{-1}} f$. Thus we obtain $\pm c_w \in \Z$, which is a contradiction.
\end{proof}

If $R\subseteq S$ is a subring and $s\in S$, we write $R[s]$ for the subring of $S$ generated by $R$ and $s$.
This will allow us to formulate the following lemma, corresponding to~\cite[Corollary~2.6]{EKL14}. Recall 
${}^\mathfrak{b}\!{\Lambda}'^{\text{--}}_{n-1} \subset  {\Pol}'^{\text{--}}_{n-1}$ from \eqref{def:shiftedb}.

\begin{lem}
	\label{lem:bsymprime}
The following identity inside $\tyodd{\Pol}_n$ holds:
$\bodd{\Lambda}_n[x_1^2] = {}^\mathfrak{b}\!{\Lambda}'^{\text{--}}_{n-1}[x_1^2]$.
\end{lem}

\begin{proof}
Recall from Theorem~\ref{thm:bodd} that $\bodd{\Lambda}_n =\Z[\elembo_1, \elembo_2, \ldots, \elembo_n]$, and similarly we have ${}^\mathfrak{b}\!{\Lambda}'^{\text{--}}_{n-1} =\Z[\elemboprime_1, \ldots, \elemboprime_{n-1}]$.
One checks by definition that, for any $k\geq 0$,  
	\begin{equation}  \label{eq:kk}
		\elemboprime_k = \sum_{j=0}^k (-1)^j x_1^{2j} \, \elembo_{k-j}.
	\end{equation}
	It follows that $\bodd{\Lambda}_n[x_1^2] \supseteq {}^\mathfrak{b}\!{\Lambda}'^{\text{--}}_{n-1}[x_1^2]$. 
Rewrite \eqref{eq:kk} as 
\[
\elembo_k =\elemboprime_k- \sum_{j=1}^k (-1)^j x_1^{2j} \elembo_{k-j}.
\]
This implies by induction on $k$ that  $\elembo_k \in {}^\mathfrak{b}\!{\Lambda}'^{\text{--}}_{n-1}[x_1^2]$, and so  $\bodd{\Lambda}_n[x_1^2] \subseteq {}^\mathfrak{b}\!{\Lambda}'^{\text{--}}_{n-1}[x_1^2]$. 
\end{proof}

\begin{prop}
	\label{prop:bfreemodule}
As a left or right $\bodd{\Lambda}_n$-module,  $\tyodd{\Pol}_n$ is a free of graded rank $q^{n^2}[2n]!!$, with a homogeneous basis given by the spin type B Schubert polynomials ${\{\schubo_w\}}_{w\in B_n}$.
\end{prop}

\begin{proof} 
	The proof below imitates the type A proof for~\cite[Proposition~2.13]{EKL14}. We shall give the detail on the left module case.
It suffices to show that multiplication map $\bodd{\Lambda}_n \otimes \mchbo_n \to \tyodd{\Pol}_n$ is an isomorphism of abelian groups. 

	To that end, we fist claim that any $f\in\tyodd{\Pol}_n$ can be expressed in the form
	\begin{equation*}
		f = \sum_{k=1}^{2n-1}\sum_j \ell_{k,j}h_{k,j} x_1^k,
	\end{equation*}
	for $h_{k,j}\in\bodd{\mathcal{H'}}_{n-1}, \ell_{k,j}\in\bodd{\Lambda}_n$.
	We prove this by induction on $n$, with the base case $n=1$ being trivial.
	Given $f\in\tyodd{\Pol}_n$, we expand in powers of $x_1$:
$f = \sum_k x_1^k f_k$, 
	for $f_k \in {\Pol}'^{\text{--}}_{n-1}$, and then use the inductive hypothesis to write
	\begin{equation*}
		f = \sum_k \sum_{i=1}^{2n-3} \sum_j x_1^k \ell_{i,j,k} h_{i,j,k} x_2^i,
	\end{equation*}
	where $h_{i,j,k}\in{}^\mathfrak{b}\!{\mathcal H}''^{\text{--}}_{n-2}$, $\ell_{i,j,k}\in{}^\mathfrak{b}\!{\Lambda}'^{\text{--}}_{n-1}$, and $f_k\in{\Pol}'^{\text{--}}_{n-1}$ for all $i,j,k$.
	Since $h_{i,j,k}x_2^i\in {}^\mathfrak{b}\!{\mathcal H}'^{\text{--}}_{n-1}$ (after moving the $x_2$, at the expense of a sign change),    using~\lemref{lem:bsymprime} we can rewrite this expression as
	\begin{equation*}
		f = \sum_{k=1}^{2n-1} \sum_j x_1^k \ell_{k,j} h_{k,j},
	\end{equation*}
where $h_{k,j}\in{}^\mathfrak{b}\!{\mathcal H}'^{\text{--}}_{n-1}$ and $\ell_{k,j}\in\bodd{\Lambda}_n$.

	The above claim implies surjectivity of the multiplication map, with injectivity following from an identical argument as for \lemref{lem:mchdo} below.
	Finally the graded rank formula follows from the identity
	\begin{equation*}
		\sum_{w\in B_n} q^{\deg{\schubo_w}} = \sum_{w\in B_n} q^{2\ell(w)} = q^{n^2}[2n]!!.
	\end{equation*}
The proposition is proved.
\end{proof}

\begin{lem}
	\label{lem:factoring}
	For $1\leq i\leq n$, $g\in\bodd{\Lambda}_n$, and $f\in\tyodd{\Pol}_n$, we have 
	\begin{equation*}
		\pbo_i(fg) = \pbo_i(f)g.
	\end{equation*}
Hence the left action of $\bodd{\NH}_n$ and the right action of $\bodd{\Lambda}_n$ on $\tyodd{\Pol}_n$  commute. 
\end{lem}

\begin{proof}
	This follows by~\eqnref{eq:leibnizb}--\eqnref{eq:leibnizb:B} and the fact that $g\in\ker (\pbo_i)$.
\end{proof}

Finally, we arrive at the main structure result for $\bodd{\NH}_n$. 
\begin{thm}
	\label{thm:biso}
We have the following $\Z$-algebra isomorphisms: 
	\begin{equation*}
	   \bodd{\NH}_n \xrightarrow{\cong} \End_{\bodd{\Lambda}_n}(\tyodd{\Pol}_n) \cong \Mat_{q^{n^2}[2n]!!}(\bodd{\Lambda}_n).
	\end{equation*}
\end{thm}

\begin{proof}
It follows by Lemma~\ref{lem:factoring} that we have an algebra homomorphism 
\[
\phi: \bodd{\NH}_n \longrightarrow  \End_{\bodd{\Lambda}_n}(\tyodd{\Pol}_n), 
\]
where $\tyodd{\Pol}_n$ is regarded as a right $\bodd{\Lambda}_n$-module. 

	The injectivity of $\phi$ follows from the faithfulness of the action of $\bodd{\NH}_n$.
	Since $\bodd{\NH}_n$ and $\End_{\bodd{\Lambda}_n}(\tyodd{\Pol}_n)$ have the same graded rank by \corref{cor:bsym2}, \corref{cor:bfaithful} and \propref{prop:bfreemodule}, $\phi$ is surjective as well.
\end{proof}

\begin{rem}  
   \label{rem:unequal}
The type B spin Hecke algebra $\mathfrak H_{B_n}^-$ defined in \cite[Definition~4.3]{KW08} has 2 parameters $u_1,u_2 \in \C.$ A spin type B nilHecke algebra of 2 parameters $\bodd{\NH}_n(u_1,u_2)$ can be defined as in Definition~\ref{def:oddtypebnh}, replacing~\eqnref{eq:arel.4}  by
		\begin{align*}
			x_i \pso_i + \pso_i x_{i+1} &= u_1, \qquad
			\pso_i x_i + x_{i+1} \pso_i = u_1,
		\end{align*}
	and~\eqnref{eq:brel.4} by
	\begin{equation*}
		\pbo_n x_n + x_n \pbo_n = u_2.
	\end{equation*}

Now assume both $u_1$ and $u_2$ are nonzero.  All constructions and results in this paper remain valid for the spin type B nilHecke algebra with 2 parameters, once we relax the base ring from $\Z$ to $\C$. This is true because the corresponding Demazure operators (of 2 parameters) can be simply obtained by a rescaling of the current ones, i.e., replacing $\pso_i$ by $u_1 \pso_i$ and $\pbo_n$ by $u_2 \pbo_n$. In particular, over the field $\C$, \thmref{thm:biso} still holds for $\bodd{\NH}_n(u_1,u_2)$. 
\end{rem}

\subsection{The spin type D case}

Let $\mchdo_n$ be the $\Z$-span of the spin type D Schubert polynomials $\{\schudo_w\}_{w\in D_n}$.
Denote by 
\[
\tyodd{\Pol}_{n,\Q} =\Q\otimes_\Z \tyodd{\Pol}_n,
\qquad
\dodd{\Lambda}_{n,\Q} =\Q\otimes_\Z \dodd{\Lambda}_{n},
\qquad
\dodd{\NH}_{n,\Q} =\Q\otimes_\Z \dodd{\NH}_{n}.
\]

\begin{lem}  
	\label{lem:mchdo}
	\quad
 \begin{enumerate}
 \item
	The Schubert polynomials $\{\schudo_w\}_{w\in D_n}$  form a $\Z$-basis for 
	$\mchdo_n$.  
 \item
	The multiplication map $\mchdo_n  \otimes\dodd{\Lambda}_n  \rightarrow \tyodd{\Pol}_n$ is injective.
\end{enumerate}
\end{lem}

\begin{proof}
We make the following

{\bf Claim.}  $\sum_{w\in D_n}  \schudo_w c_w =0$ for $c_w \in  \dodd{\Lambda}_{n,\Q}$ if and only if $c_w=0$ for all $w$.

Indeed, assume $\sum_{w\in D_n}  \schudo_w c_w =0$ and $u\in D_n$ is a maximal length element such that $c_u\neq 0$. It follows by \propref{prop:schudo}, \eqref{eq:Dss} and  the Leibniz rule that 
\[
2^{n-1} c_u = \pdo_u (\sum_{w\in D_n} \schudo_w c_w) =0, 
\]
and so $c_u=0$, a contradiction. The Claim is proved. 

Both (1) and (2) follows from this Claim.
\end{proof}

%

\begin{prop}
	\label{prop:dfreemodule}
As a right (or  a left) $\dodd{\Lambda}_{n,\Q}$-module,  $\tyodd{\Pol}_{n,\Q}$ is free of graded rank $q^{n(n-1)}[n][2n-2]!!$ with a homogeneous basis given by the spin type D Schubert polynomials $\{\schudo_w\}_{w\in D_n}$.
\end{prop}

\begin{proof}
The two cases are similar, and let us choose to prove the right module case.  

It follows by \lemref{lem:mchdo}(2) that the multiplication map $\mchdo_n  \otimes\dodd{\Lambda}_n \rightarrow \tyodd{\Pol}_n$ is injective. The surjectivity of this map follows by comparing the graded ranks, using \corref{cor:rkSFDn} and \lemref{lem:mchdo}(1). Therefore the proposition follows by noting that $\rk_q (\mchdo_n) = q^{n(n-1)}[n][2n-2]!!$. 
 \end{proof}

The following is a type D analogue of~\lemref{lem:factoring} with the same proof.
\begin{lem}
	\label{lem:factoringd}
	For $1\leq i\leq n$, $g\in\dodd{\Lambda}_n$, and $f\in\tyodd{\Pol}_n$, we have
	$\pdo_i(fg) = \pdo_i(f)g.$
\end{lem}

\begin{thm}
	\label{thm:diso}
We have the following algebra isomorphisms:
	\begin{equation*}
\dodd{\NH}_{n,\Q}  \xrightarrow{\cong} \End_{\dodd{\Lambda}_{n,\Q}} (\tyodd{\Pol}_{n,\Q}  ) \cong \Mat_{q^{n(n-1)}[n][2n-2]!!}  (\dodd{\Lambda}_{n,\Q} ).
	\end{equation*}
\end{thm}

\begin{proof}
It follows by \lemref{lem:factoringd} that we have an algebra homomorphism 
\[
\phi: \dodd{\NH}_n \longrightarrow  \End_{\dodd{\Lambda}_n}(\tyodd{\Pol}_n), 
\]
where $\tyodd{\Pol}_n$ is regarded as a right $\dodd{\Lambda}_n$-module. 
	The injectivity of $\phi$ follows from the faithfulness of the action of $\dodd{\NH}_n$.
	Since $\dodd{\NH}_n$ and $\End_{\dodd{\Lambda}_n}(\tyodd{\Pol}_n)$ have the same graded rank by 
   \corref{cor:rkSFDn}, \corref{cor:dfaithful} and \propref{prop:dfreemodule}, $\phi$ is surjective as well.
\end{proof}

Denote by $\dodd{\NH}_n[\frac12] = \Z[\frac12] \otimes_\Z \dodd{\NH}_n$ after a base change, and so on. The following conjecture has been verified for $n=2$.    

\begin{conj}
  \label{conj:Z12}
We have the following graded algebra isomorphism
	\begin{equation*}
\dodd{\NH}_n[\frac12]  \xrightarrow{\cong} \End_{\dodd{\Lambda}_n [\frac12]} \Big(\tyodd{\Pol}_n[\frac12] \Big) \cong \Mat_{q^{n(n-1)}[n][2n-2]!!} \Big(\dodd{\Lambda}_n[\frac12] \Big).
	\end{equation*}
\end{conj}

\subsection{Categorification}

We consider  the category $\bodd{\NH}_{n}$-pmod (and respectively, $\aodd{\NH}_{n}$-pmod) of finitely generated $\Z\times \Z_2$-graded left projective $\bodd{\NH}_{n,\Q}$-modules (respectively, $\aodd{\NH}_{n,\Q}$-modules)
and its Grothendieck group $K_0(\bodd{\NH}_n)$ (respectively, $K_0(\aodd{\NH}_n)$). 
The category $\bodd{\NH}_{n}$-pmod admits a $\Z$-grading shift functor $\mathbf q$ and a parity shift functor $\Pi$. 
Define 
\begin{align*}
   K_0(\aodd{\NH}) &= \bigoplus_{n \geq 0} K_0(\aodd{\NH}_n),  \qquad
   K_0(\bodd{\NH}) = \bigoplus_{n \geq 0} K_0(\bodd{\NH}_n). 
\end{align*}
Via the natural inclusion of algebras $\aodd{\NH}_m \otimes \aodd{\NH}_n\rightarrow \aodd{\NH}_{m+n}$, one defines the induction and restriction functors (for varying $m,n$), which give rise to an induction functor and a restriction functor on the Grothendieck group level as follows:
\begin{align*}
\aeven{\IND}: K_0(\aodd{\NH})  \otimes K_0(\aodd{\NH}) \longrightarrow K_0(\aodd{\NH}), 
\\
\aeven{\RES}: K_0(\aodd{\NH}) \longrightarrow K_0(\aodd{\NH})  \otimes K_0(\aodd{\NH}).
\end{align*}
These functors equip a twisted bialgebra structure on $K_0(\aodd{\NH})$ ~\cite{EKL14, HW15} (also cf. \cite{La08}). 

One can introduce a bar map $\ov{\phantom{x}}$ on  $K_0(\aodd{\NH})$ which satisfies $\ov{q}=\pi q^{-1}$, $\ov{\pi}=\pi$; cf. \cite{HW15};  the $(q,\pi)$-integers \eqref{eq:doublefactorial} are bar-invariant. The category $\aodd{\NH}_n$-pmod contains a unique (up to  isomorphism) self-dual projective indecomposable module $\mathcal E^{(n)}$. Hence, we have 
\[
K_0(\aodd{\NH}_n) \cong \Z[q,q^{-1}, \pi],
\qquad
K_0(\aodd{\NH}) \cong \oplus_{n\ge 0} \Z[q,q^{-1}, \pi] \mathcal E^{(n)}.
\] 
The twisted bialgebra $K_0(\aodd{\NH})$ is identified with the half quantum covering algebra of rank one $\U^+_{q,\pi}(\mathfrak{sl}_2)$; cf. \cite{HW15}. Recall $\U^+_{q,\pi}(\mathfrak{sl}_2) =\oplus_{n\in \N} \Z[q,q^{-1},\pi] E^{(n)}$ is an algebra 
such that 
\[
E^{(m)}  E^{(n)} = \frac{[m+n]_\pi !}{[m]_\pi ! \cdot [n]_\pi !} E^{(m+n)}.
\]
If we ignore the $\Z_2$-grading in the above considerations (which correspond to setting $\pi=1$), then the twisted bialgebra $K_0(\aodd{\NH})$ is identified with the half quantum group  of rank one $\U^+_{q}(\mathfrak{sl}_2)$, as first shown in \cite{EKL14}. 

Now we consider the type B algebras as well as type A algebras.
There exist natural inclusion of algebras 
\[
\aodd{\NH}_m \otimes \bodd{\NH}_n\longrightarrow \bodd{\NH}_{m+n}, 
\]
 and this gives rise to an induction functor and a restriction functor on the Grothendieck group level as follows:
\begin{align*}
\IND: K_0(\aodd{\NH})  \otimes K_0(\bodd{\NH}) \longrightarrow K_0(\bodd{\NH}), 
\\
\RES: K_0(\bodd{\NH}) \longrightarrow K_0(\aodd{\NH})  \otimes K_0(\bodd{\NH}).
\end{align*}

\begin{prop}
\label{prop:module}
The functors $\IND, \RES$ equip  $K_0(\bodd{\NH})$ with a bialgebra module structure over the twisted bialgebra $K_0(\aodd{\NH})$. 
\end{prop}
We skip the proof of the above proposition, which does not really differ from the proof for the bialgebra structure on $K_0(\aodd{\NH})$ in \cite{La08, EKL14, HW15}. 

The category $\bodd{\NH}_n$-pmod contains a unique (up to  isomorphism) self-dual projective indecomposable module $\beven{\mathcal E}^{(n)}$, which is isomorphic to (up to some grading shift) the polynomial representation $\Pol_n^-$. Hence we have
\[
K_0(\bodd{\NH})  \cong \bigoplus_{n\ge 0} \Z[q,q^{-1}, \pi] \beven{\mathcal E}^{(n)}. 
\]
 Denote by $\beven{\mathcal E}^{n}$ the regular representation of $\bodd{\NH}_n$. Recalling $[n]_{\pi}!!$ from \eqref{eq:doublefactorial}, we have
\[
\beven{\mathcal E}^{n} \cong \bigoplus_{[n]_{\pi}!!} \beven{\mathcal E}^{(n)}.
\]
Here $\bigoplus\limits_f M$, for a Laurent polynomial $f =\sum_{j,\alpha} f_{j,\alpha} q^j \pi^\alpha \in \N[q,q^{-1},\pi]$ and a graded module $M$, denote the direct sum of $f_{j,\alpha}$ copies of $\mathbf{q}^j \Pi^\alpha M$.
Thus the left $K_0(\aodd{\NH})$-module structure on $K_0(\bodd{\NH})$ is given by 
\[
{\mathcal E}^{(m)} \, \beven{\mathcal E}^{(n)} =\qbinom{m+n}{n}_{\bb} \beven{\mathcal E}^{(m+n)}, \;\; 
\text{ where } \qbinom{m+n}{n}_{\bb}= \frac{[2m+2n]_\pi !!}{[m]_\pi ! \cdot [2n]_\pi !!}.
\] 

\begin{rem}
A similar categorification in type D can be formulated as above. 
\end{rem}


\appendix
\section{NilHecke algebras are matrix algebras}
\label{appendix:A}

In this appendix, we review the polynomial representation of the nilHecke algebra $\NH_W$ associated to any Weyl group via Demazure operators. We show that the algebra $\NH_{W,\Q}$ over $\Q$ is a matrix algebra with entries in the the algebra of $W$-invariant polynomials. 

\subsection{The polynomial representations}

Let $W$ be a finite Weyl group generated by simple reflections $s_i$  $(i\in I)$, and $\mathfrak h$ be the reflection representation (over $\Z$) of $W$. Let $\alpha_i \in \h^*$ denote the simple root and  let $\alpha_i^\vee \in \mathfrak h$ be the simple coroot, for $i\in I$. The degenerate affine Hecke algebra $\hh_W$ associated to $W$ was introduced by Lusztig \cite{Lu89}, and its corresponding nilHecke agebra over $\Z$ will be denoted by $\NH_W$. 

As a $\Z$-module, $\NH_W \cong \NC_W \otimes S(\mathfrak h)$, and $\NH_W$ contains the nilCoxeter algebra $\NC_W =\Z \langle \partial_i, i \in I \rangle$ (with $\partial_i^2=0$) and the symmetric algebra $S(\mathfrak h)$ as $\Z$-subalgebras. In addition, it satisfies the following relations:
\begin{align}  \label{H:dual}
x \, \partial_i -\partial_i \,  x^{s_i} = \langle x, \alpha_i  \rangle,
\qquad \text{  for } x\in \h, i\in I, 
\end{align}
where $x^{s_i}$ denotes the image of $x$ under the reflection $s_i$. One sometimes multiplies the RHS of \eqref{H:dual} by a parameter  $u_i$ depending on the $W$-conjugacy classes of $s_i$, and $u_i$ is normalized to be 1 in this paper.
Note $\NH_W$ is a $\Z$-graded algebra with $|\partial_i|=-2$ and $|h|=2$, for all $i\in I$ and $h\in \mathfrak h$. Alternatively, there is a natural $\Z$-filtered algebra structure on $\hh_W$ and its associated graded is isomorphic to $\NH_W$. 
The following fact is folklore. 

\begin{prop}  \label{prop:polR}
The nilHecke agebra $\NH_W$ admits a polynomial representation $\Pol_{\mathfrak h} =S(\mathfrak h)$, where $h\in \mathfrak h$ acts as a multiplication operator and $\partial_i$ acts as the Demazure operator $\cdot \frac{1-s_i}{\alpha_i}$. 
\end{prop}

\begin{proof}
Define $\Pol_{\mathfrak h}$ to be the induced $\hh$-module $\text{Ind}_{\NC_W}^\hh \Z$ from the trivial $\NC_W$-module $\Z$. The rest follows. 
\end{proof}

An argument almost identical to Lemma~\ref{lem:bexact} shows that $\ker(\partial_i) =\im(\partial_i) \subset S(\mathfrak h)$. 
Define
\[
\Lambda_{\mathfrak h} := \bigcap_{i\in I} \ker(\partial_i).
\]
It follows from Proposition~\ref{prop:polR} that 
\begin{align}  \label{eq:inv}
\Lambda_{\mathfrak h} =S(\mathfrak h)^W, 
\end{align}
the subalgebra of $W$-invariants in $S(\mathfrak h)$. 

\subsection{The nilHecke algebras of classical type}

Let us be more explicit for classical type. 
For $W$ of classical types $A_{n-1}, B_n, D_n$, we naturally identify $S(\mathfrak h)$ with the polynomial algebra $\Z[\x_1, \ldots, \x_n]$, and write $\tya{\NH}_n =\NH_{{S_{n}}}$, $\tyb{\NH}_n =\NH_{{B_n}}$, $\tyd{\NH}_n =\NH_{{D_n}}$. 

\begin{rem}
The only difference between presentations of $\tyb{\NH}_n$ (and respectively, $\tyd{\NH}_n$) and $\bodd{\NH}_n$  in Definition~\ref{def:oddtypebnh} (and respectively, $\dodd{\NH}_n$ in Definition ~ \ref{def:oddtypednh}) are some suitable sign changes; such a phenomenon has already been observed between degenerate affine Hecke algebras of classical type and their spin counterparts \cite{KW08}.
\end{rem}

In classical types, let us write $\Lambda_{\mathfrak h}$ in terms of the more familiar notations $\tya{\Lambda}_n$, $\tyb{\Lambda}_n$, $\tyd{\Lambda}_n$, and so forth; we also write $\tyeven{\Pol}_{\mathfrak h}$ as $\tyeven{\Pol}_n$. 
The equality \eqref{eq:inv} for classical types reads
\begin{equation}  \label{eq:invB}
\tya{\Lambda}_n = \Z[\x_1, \ldots, \x_n]^{S_n},
\quad
\tyb{\Lambda}_n = \Z[\x_1^2, \ldots, \x_n^2]^{S_n},
\quad
\tyd{\Lambda}_n = 
\tyb{\Lambda}_n [\x_1\cdots \x_n].
\end{equation}
Introduce, for $1 \leq k \leq n-1,$
\begin{align}
	\beven{\varepsilon}_k (\x_1,\dotsc,\x_n) &= \deven{\varepsilon}_k (\x_1,\dotsc,\x_n) =  
		\sum_{1 \leq i_1 < \dotsb < i_k \leq n} \x_{i_1}^2\dotsm \x_{i_k}^2,  
\\
	\beven{\varepsilon}_n (\x_1,\dotsc,\x_n) &= \x_1^2 \dotsm \x_n^2, 
\quad\quad 
   \deven{\varepsilon}_n (\x_1,\dotsc,\x_n) = \x_1 \dotsm \x_n. %
\end{align}
By Chevalley's theorem, $\tyb{\Lambda}_n$ and $\tyd{\Lambda}_n$ are polynomial algebras in $n$ generators, and  
\begin{equation} 
  \label{eq:invBD2}
\tyb{\Lambda}_n = \Z[\beven{\varepsilon}_1, \ldots, \beven{\varepsilon}_{n-1}, \beven{\varepsilon}_n],
\qquad
\tyd{\Lambda}_n = \Z[\deven{\varepsilon}_1, \ldots, \deven{\varepsilon}_{n-1}, \deven{\varepsilon}_n].
\end{equation}

We record the following corollary of \eqref{eq:invBD2}.
\begin{cor}
	\label{cor:blambdarkeven}
The graded ranks of $\beven{\Lambda}_n$ and $\deven{\Lambda}_n$ are given as follows:
	\begin{align}
	\begin{split}
		\rk_q(\beven{\Lambda}_n) &= q^{-n^2}\frac{1}{{(1-q^2)}^n}\frac{1}{[2n]!!},
		\\
		\rk_q(\deven{\Lambda}_n) &= q^{-n(n-1)} \frac{1}{(1-q^2)^n} \frac{1}{[n][2n-2]!!}. 
	\end{split}
	\end{align}
\end{cor}

\subsection{The matrix algebra identification}
Denote by $\rk_q W = \sum_{w\in W} q^{2\ell(w)}$ the Poincare polynomial of $W$. It is equal to the graded rank for $\NC_W$, upon the replacement of $q$ by $q^{-1}$. 
Note that the action of $\NH_W$ on $\Pol_{\mathfrak h}$ in Proposition~\ref{prop:polR} induces a $\Z$-algebra homomorphism $\NH_W \longrightarrow \End_{{\Lambda}_{\mathfrak h}}(\tyeven{\Pol}_{\mathfrak h})$. 
The following result  in type A was established in \cite{La08} over $\Z$. We add a subscript $\Q$ to indicate the base change from $\Z$ to $\Q$, writing $\NH_{W,\Q}$ and so on. 

\begin{thm}
	\label{thm:iso}
We have the following isomorphism of graded algebras:    
\[
\NH_{W,\Q} \xrightarrow{\cong}   \End_{{\Lambda}_{\mathfrak h,\Q}}(\Pol_{\mathfrak h,\Q}) 
	 \cong \Mat_{\mathrm{rk}_qW}(\Lambda_{\mathfrak h,\Q}).
\]
%
\end{thm}

\begin{proof}
%
We give two proofs below: an algebraic argument in (a), and (Webster and Shan)  a geometric argument in (g).

(a). The general $W$ case follow by the same type of arguments for \thmref{thm:diso} in spin type D, with key input being \propref{prop:dfreemodule} (which is in turn based on \propref{prop:schudo}). Recall $\Pol_{\mathfrak h,\Q}=S(\mathfrak h_\Q)$. Let $w_0$ denotes the longest element in $W$. Thus for the argument to go through in the general $W$ case, all we need is the following.

{\bf Claim.} There exists  $\mathfrak{s}_{w_0} \in S(\mathfrak h_\Q)$ such that $\partial_{w_0} (\mathfrak{s}_{w_0})$ is a nonzero constant. 

Let us prove the Claim. 
We use the following well known facts over a field of characteristic zero, cf. \cite{Hi82}. The module $S(\mathfrak h_\Q)$ is free over the algebra $S(\mathfrak h_\Q)^W$, with a basis given by any lift of a basis for the coinvariant algebra $S(\mathfrak h_\Q)_W$. Let $\mathfrak{s}_{w_0}$ be a homogeneous lift of a highest degree element $p_{w_0}$ 
in  $S(\mathfrak h_\Q)_W$ (for example, $p_{w_0}$ can be the Schubert class of a point in the identification $S(\mathfrak h_\Q)_W$ with the cohomology ring of a flag variety $G/B$ of corresponding type). 
By \cite{BGG73, De74}, we have $\partial_{w_0}(p_{w_0}) \neq 0$.
Thus $\partial_{w_0} (\mathfrak{s}_{w_0}) \neq 0$, and for degree reasons, $\partial_{w_0} (\mathfrak{s}_{w_0})$ must be a constant. 

(g). When the second author showed the isomorphism in \thmref{thm:iso}(2) to Peng Shan and Ben Webster some time ago, they separately supplied a geometric argument, which is sketched as follows.

Let $G$ be a simple algebraic group with a Borel subgroup $B$ and Weyl group $W$. The nilHecke algebra $\NH_W$ is the $G$-equivariant homology of $G/B \times G/B$ endowed with convolution product (the BGG-Demazure operator for $w\in W$ corresponds to the fundamental class of the orbit closure associated to $w$).  This is the same as the Ext-algebra $\text{Ext}^*(\pi_* \underline{\Q}, \pi_* \underline{\Q})$  in the $G$-equivariant derived category (cf. \cite{CG97}), where $\pi_* \underline{\Q}$ denotes the pushforward of the constant sheaf with $\pi: G/B \rightarrow \mathrm{pt}$.  The latter can be identified as  the algebra of endomorphisms of $H_G^*(G/B)$ over $H_G^*(\mathrm{pt})$, which is a matrix algebra, since $G/B$ is equivariantly formal.  
\end{proof}

\begin{rem}
 \label{rem:integral}
One can show a variant of Theorem~\ref{thm:iso}  for type B over $\Z$, following the proof of \thmref{thm:biso}; and also for type A, see \cite{La08}. 
The counterpart of Conjecture~\ref{conj:Z12} (if proven) provides an integral version of Theorem~\ref{thm:iso} for type D over the ring $\Z[\frac12]$. According to Webster, the geometric argument can be strengthened to work over any subring of $\Q$ in which the torsion primes of $G$ are invertible.
\end{rem}


\end{document}